\documentclass{amsart}
\usepackage{amssymb,amscd}  

\DeclareRobustCommand{\looongrightarrow}{%
  \DOTSB\relbar\joinrel\relbar\joinrel\relbar\joinrel\rightarrow
}

\newcommand{\etalchar}[1]{$^{#1}$}

\providecommand{\MR}{\relax\ifhmode\unskip\space\fi MR }

\usepackage[dvipsnames,pdftex]{xcolor}

\usepackage[hypertexnames=true,plainpages=false,colorlinks=true]{hyperref}
\providecolor{DarkBlue}{rgb}{0,0,.545}
\providecolor{DarkGreen}{rgb}{0,.392,0}
\hypersetup{citecolor=DarkGreen}
\hypersetup{linkcolor=DarkBlue}
\hypersetup{urlcolor=DarkBlue}
\usepackage[nameinlink,capitalise]{cleveref}

\newtheorem{theorem}{Theorem}
\numberwithin{theorem}{section} 
\newtheorem{proposition}[theorem]{Proposition}
\newtheorem{corollary}[theorem]{Corollary}  
\newtheorem{lemma}[theorem]{Lemma}
\newtheorem{conjecture}[theorem]{Conjecture}
\theoremstyle{definition}
\theoremstyle{definition}\newtheorem{remark}[theorem]{Remark}
\theoremstyle{definition}\newtheorem{example}[theorem]{Example}
\theoremstyle{definition}
\theoremstyle{definition}
\theoremstyle{definition}\newtheorem*{definition*}{Definition}  
\newtheorem*{theorem*}{Theorem}
\newtheorem*{lemma*}{Lemma}
\newtheorem*{prop*}{Proposition}
\newtheorem*{prob*}{Problem}
\newtheorem*{coro*}{Corollary}  
\numberwithin{equation}{section}

\usepackage{algorithmic}

\usepackage[ruled]{algorithm}

\newcommand\Q{\mathbb{Q}}
\newcommand\Z{\mathbb{Z}}
\newcommand\F{\mathbb{F}}
\newcommand{\calC}{\mathcal{C}}
\newcommand{\calD}{\mathcal{D}}
\newcommand{\J}{\mathfrak{J}}
\newcommand\Lk[1]{\mathcal{L}^{#1}}
\newcommand\K{\mathcal{K}}

\newcommand\Projective{{\bf P}} 
\newcommand\kgen{\mathcal{K}^{\hbox{\scriptsize gen}}} 
\newcommand\hkgen{\hat{\mathcal{K}}^{\hbox{\scriptsize gen}}} 
\newcommand\crosen{\mathcal{C}_{\hbox{\scriptsize rosen}}} 
\newcommand\ksqr{\mathcal{K}^{\hbox{\scriptsize sqr}}} 
\newcommand\kgaudry{\mathcal{K}^{\hbox{\scriptsize fast}}}

\newcommand\spaceR{{X_R}}
\newcommand\spaceS{{X_S}}
\newcommand\spaceRS{{X_{R,S}}}
\newcommand\spaceRi[1]{{X^{(#1)}_{R}}}
\newcommand\spaceSi[1]{{X^{(#1)}_{S}}}
\newcommand\spaceRSi[1]{{X^{(#1)}_{R,S}}}
\newcommand\basisR{B_R}
\newcommand\basisS{B_S}
\newcommand\basisRS{B_{R,S}}
\newcommand\basisRi[1]{{B^{(#1)}_R}}
\newcommand\basisSi[1]{{B^{(#1)}_S}}

\newcommand\cost[1]{\textsf{Cost}_{\text{#1}}}

\usepackage{colonequals}
\usepackage{url}
\usepackage{tikz}
\usepackage{pgfplots}
\usepackage{tikz-cd}
\usepackage{adjustbox}
\usepackage{multirow}

\begin{document}

\title{Isogenies on Kummer Surfaces}

\author{Maria Corte-Real Santos}
\address{University College London, London WC1E 6BT, United Kingdom}
\email{maria.santos.20@ucl.ac.uk}

\author{E. Victor Flynn}
\address{Mathematical Institute, University of Oxford, Andrew Wiles
Building, Radcliffe Observatory Quarter, Woodstock Road, Oxford
OX2 6GG, United Kingdom}
\email{flynn@maths.ox.ac.uk}

\thanks{21 September, 2024}

\begin{abstract}
We first give a cleaner and more direct approach to the derivation
of the Fast model of the Kummer surface. We show how
to construct efficient $(N,N)$-isogenies, for any odd~$N$,
both on the general Kummer surface and on the Fast model.
\end{abstract}

\maketitle

\section{Introduction}
\label{section1}
Various models have been constructed of the Kummer surface 
related to curves of genus~$2$, with more recent work
emphasising models which are amenable to more efficient
computations. There has also been considerable recent
interest in computing isogenies on these surfaces. The intention of
this article is to first give a more explicit derivation
of the \emph{Fast} Kummer surface model (a model introduced by
Gaudry~\cite{GaudryFastGenus2} in the context of cryptography,
and explored algorithmically by Chudnovsky and 
Chudnovsky~\cite{Chudnovsky} which allows for particularly 
fast computations), and then to construct~$(N,N)$-isogenies 
on this model, for any odd~$N$.
\par  
In~\Cref{section2}, we recall the General Kummer model 
of the Kummer surface in $\Projective^3$,
given on p.18 of~\cite{CasselsFlynnBook}, and summarise
a number of associated ideas, such as the linear map induced
by adding a point of order~$2$, certain biquadratic forms
which relate to the group law on the Jacobian variety,
and the Richelot isogeny. 

In~\Cref{section3}, we summarise the 
\emph{Squared} Kummer model described by Cosset~\cite[Ch. 4]{CossetThesis},
and the linear map which relates it to the General Kummer
of the Rosenhain curve given by Chung, Costello 
and Smith~\cite[\S 3]{ChungCostelloSmithFastUniform}. 
In~\Cref{section4}, we summarise
the Fast Kummer surface model, described by Gaudy~\cite{GaudryFastGenus2},
which also allows fast computations and is~$(2,2)$-isogenous
to the model in~\Cref{section3}. The derivation
by Gaudry~\cite{GaudryFastGenus2} uses several identities
of theta functions as justification, an extension of the ground
field was used, and there was no explicit Jacobian of which
this was the Kummer Surface, but instead a $(2,2)$-isogenous
Jacobian was given over the field extension. 
The first main contribution of this article is that, in~\Cref{section4},
we give an entirely algebraic rederivation of the
Fast Kummer surface $\kgaudry$ by finding a
curve~$\calD$ together with a linear
map  
from the General Kummer surface of~$\calD$
in~\Cref{curveD} to the Fast Kummer in~\Cref{Gaudryeqn}.
Our aim is that this will make both the derivation and
application of this elegant form more accessible to a wider audience
who may use it over number fields for such things as height
constants and descents. Everything is performed over the ground field
of the parameters,
and we now have an explicit Jacobian (namely the Jacobian variety
of~$\calD$) of which $\kgaudry$ is the Kummer surface.

In~\Cref{section5}, we describe a general method to construct 
$(N,N)$-isogenies between Kummer surfaces with efficiently computable 
biquadratic forms, for any odd~$N$, using purely algebraic methods. 
More specifically, we show how to 
obtain the $(N,N)$-isogeny as a composition 
$\varphi = \lambda \circ \psi$, where $\psi$ has the desired kernel 
and $\lambda$ is a linear map which moves the image into the 
correct Kummer form. We give algorithms to compute $\psi$ for 
any Kummer surface model with efficiently computable biquadratic forms. 
In the case of General and Fast Kummer surfaces, we explicitly 
show how to obtain this final linear map $\lambda$ 
in~\Cref{subsec:isogenies-general,subsec:isogenies-fast}, respectively.
This culminates in two further 
main contributions: an algorithm \textsf{GetIsogeny} 
(see~\Cref{alg:getisogeny} in~\Cref{section5}) which
recovers the explict formul\ae{} describing the $(N,N)$-isogeny 
$\varphi$ from the 
$N$-torsion points generating the kernel; and an algorithm \textsf{GetImage}
(see~\Cref{alg:getimage} in~\Cref{section5}) which computes the image of the isogeny $\varphi$.
In particular, we considerably refine
the scaling step of the algorithm by exploiting
maps and symmetries on $\kgaudry$ which allow us to avoid expensive
square root operations in the ground field.
These explicit algebraic derivations and descriptions will be of
potential use not only
in the Cryptographic community, but also to those researchers in
Arithmetic Geometry who which to perform descent via isogeny
over number fields.
Throughout, we provide concrete complexities for all 
our algorithms, highlighting the main bottleneck of our work: 
finding degree-$N$ homogeneous forms that are invariant under 
translation by an $N$-torsion point on the Kummer surface. 
\par
Since the Fast Kummer surfaces yield the most efficient 
$(N,N)$-isogenies, 
in~\Cref{section6} we provide concrete timings for running
our methods on this model for odd primes $N \leq 19$.
\par
The software accompanying this paper is written in {\tt MAGMA}
~\cite{magma} and Maple \cite{maple},
and is publicly available under the MIT license. It is
available at 
\begin{center}
\url{https://github.com/mariascrs/NN_isogenies}.
\end{center}
\par
\noindent
\subsection{Comparison to other methods} 
We briefly compare our methods
for computing $(N,N)$-isogenies between different models of 
Kummer isogenies for odd $N$ to those in previous literature. 
The case $N = 3$ and $5$ for the General Kummer surface was 
determined by Bruin, Flynn and Testa~\cite{BruinFlynnTesta33Isogenies}
and by~\cite{Flynn55Isogenies}, respectively. We 
recover these isogenies in~\Cref{section5}. 
Turning to a different model, Corte-Real Santos, Costello and 
Smith~\cite{CorteRealSantosCostelloSmithIsogenies} give 
optimised formul\ae{} for
$(3,3)$-isogenies between Fast Kummer surfaces, and touch 
on how these methods could be generalised to any odd $N$.
We then extend these ideas to give
a method for computing $(N,N)$-isogenies between 
Fast Kummer surfaces for any odd $N \geq 5$. 
\par
Bisson, Cosset, Lubicz, and Robert launched an ambitious program
~\cite{BissonThesis,CossetThesis,LubiczRobertComputing,LubiczRobertArithmetic,RobertThesis10,RobertThesis21,LubiczRobertFast} 
based on the theory of theta functions~\cite{Mumford}
to provide asymptotically
efficient algorithms for arbitrary odd $N$ (and beyond dimension $2$ to arbitrarily high
dimension). {\tt AVIsogenies}~\cite{AVIsogenies} is a software package written in 
{\tt MAGMA} based on their results, and 
is publicly available. 

We take a different perspective on the problem and analyse 
to what extent the method depicted by Corte-Real Santos, Costello and 
Smith~\cite{CorteRealSantosCostelloSmithIsogenies} can be optimised. 
Our approach allows us to describe algorithms that output the isogeny formul\ae{}
as well as the image Kummer surface, relying solely on simple linear algebra.
To the best of our knowledge, the line of work relying on theta functions 
does not recover the formul\ae{} describing the $(N,N)$-isogeny, which 
is interesting in its own right. In this sense, 
the purposes of {\tt AVIsogenies} and our algorithm
are somewhat different. Indeed, {\tt AVIsogenies} inputs an initial Kummer
surface, isogeny kernel and a point, and outputs
the image point and (theta null point of) the image Kummer surface; 
our algorithm inputs an initial Kummer surface
and isogeny kernel, and outputs the explicit defining equations
of the image Kummer surface and of the isogeny (which can then
incidentally be used, if desired, to find images of specific
points).
\par
Though the algorithms developed using theta functions
boast better asymptotic complexity, preliminary experimental evidence suggest that our 
software outperforms {\tt AVIsogenies} for small odd $N$ 
(see~\Cref{section6} for further details). However,
for a precise and fair comparison between implementations, 
exact operation counts are needed.
\par
Furthermore, our methods produce $K$-rational $(N,N)$-isogeny 
formul\ae{} for \emph{any} Kummer surface model given $K$-rational kernel 
generators (provided it has efficiently computable biquadratic forms).
In this way, we do not require full $K$-rational $2$-torsion in order
to have a $K$-rational theta structure.

\subsection{Acknowledgements}
We thank Craig Costello and Sam Frengley for many fruitful discussions throughout the preparation 
of this paper.
We further thank Benjamin Smith for discussions that lead to the statement of~\Cref{prop:dim-and-basis}
and~\Cref{conj:basis-space}.
We are greatful to Kamal Khuri-Makdisi for kindly explaining a 
proof to~\Cref{prop:dim-and-basis}, 
and allowing us to include it in this article.
The first author was supported by UK EPSRC grant EP/S022503.

\bigskip\bigskip\bigskip

\section{Generalities on Kummer surfaces}\label{section2}
Let~$K$ be any field (not of characteristic~$2$) and
consider a general curve of genus~$2$
\begin{equation}\label{genus2general}
y^2 = \mathfrak{F}(x) 
= f_6 x^6 + f_5 x^5 + f_4 x^4 + f_3 x^3 + f_2 x^2 + f_1 x + f_0,
\end{equation}
defined over~$K$.
We represent elements of the Jacobian variety by
$\{ (x_1,y_1), (x_2,y_2) \}$, as a shorthand for the divisor
class of $(x_1,y_1) + (x_2,y_2) - \infty^+ - \infty^-$,
where~$\infty~^+$ and~$\infty^-$ denote the points on the
non-singular curve that lie over the singular point at infinity.
The Kummer surface has
an embedding (see p.18 of~\cite{CasselsFlynnBook})
in $\Projective^3$ given by $(k_1,k_2,k_3,k_4)$, where 
\begin{equation}\label{kummercoords}
k_1 = 1,\  k_2 = x_1 + x_2 ,\  k_3 = x_1 x_2,\
k_4 = (F_0 (x_1 ,x_2 )-2y_1 y_2 )/(x_1 -x_2 )^2,
\end{equation}
\noindent and where
\begin{equation}\label{F0defn}
\begin{split}
F_0 (x_1 ,x_2 )= &2 f_0 + f_1 (x_1 +x_2 ) + 2 f_2 (x_1 x_2 )
+ f_3 (x_1 x_2) (x_1 +x_2 ) \\
&+ 2 f_4 (x_1 x_2 )^2 + f_5 (x_1 x_2 )^2 (x_1 + x_2 )
+2 f_6 (x_1 x_2 )^3.
\end{split}
\end{equation}
The defining equation of the Kummer surface is given by 
\begin{equation}\label{kummerequation}
\kgen : F_1(k_1,k_2,k_3) k_4^2 + F_2(k_1,k_2,k_3) k_4 + F_3(k_1,k_2,k_3) = 0,
\end{equation}
where $F_1$, $F_2$, $F_3$ are given by:
\begin{eqnarray*}
F_1(k_1,k_2,k_3) &=& k_2^2 - 4 k_1 k_3, \\
F_2(k_1,k_2,k_3) &=&
-2 \left( 2 k_1^3 f_0 + k_1^2 k_2 f_1 + 2 k_1^2 k_3 f_2 + k_1 k_2 k_3 f_3
+ 2 k_1 k_3^2 f_4  \right. \\
  &~& \left. + k_2 k_3^2 f_5 + 2 k_3^3 f_6 \right), \\
F_3(k_1,k_2,k_3) &=&
- 4 k_1^4 f_0 f_2 + k_1^4 f_1^2 - 4 k_1^3 k_2 f_0 f_3 - 2 k_1^3 k_3 f_1 f_3
- 4 k_1^2 k_2^2 f_0 f_4 \\
&~& + 4 k_1^2 k_2 k_3 f_0 f_5 - 4 k_1^2 k_2 k_3 f_1 f_4 - 4 k_1^2 k_3^2 f_0 f_6
+ 2 k_1^2 k_3^2 f_1 f_5 \\
&~& - 4 k_1^2 k_3^2 f_2 f_4  + k_1^2 k_3^2 f_3^2 - 4 k_1 k_2^3 f_0 f_5
+ 8 k_1 k_2^2 k_3 f_0 f_6 - 4 k_2^4 f_0 f_6 \\
&~& - 4 k_1 k_2^2 k_3 f_1 f_5 + 4 k_1 k_2 k_3^2 f_1 f_6
- 4 k_1 k_2 k_3^2 f_2 f_5 - 2 k_1 k_3^3 f_3 f_5 \\
&~& - 4 k_2^3 k_3 f_1 f_6 -4 k_2^2 k_3^2 f_2 f_6
-4 k_2 k_3^3 f_3 f_6 - 4 k_3^4 f_4 f_6 + k_3^4 f_5^2.
\end{eqnarray*}
We shall refer to this model as $\kgen$, the General Kummer model.
We also recall from p.65 of~\cite{CasselsFlynnBook} that there
are local parameters~$s_1,s_2$ on the Jacobian variety, and
all coordinates of the Kummer surface can be written as formal
power series in these paramaters. We refer the reader
to~\cite{CasselsFlynnBook} for the definition of~$s_1,s_2$,
which we shall not require here. We merely note, for future reference,
that the coordinates of the General Kummer model are as follows up
to the degree~$6$ terms, when written as formal power series in
the local parameters.
\begin{equation}\label{k1k2k3k4ps}
\begin{split}
k_1 = &s_2^2 - f_2 s_2^4 - f_6 s_1^4 + 4 f_0 f_4 s_2^6 + 8 f_0 f_5 s_1 s_2^5
   + 18 f_0 f_6 s_1^2 s_2^4 - f_1 f_3 s_2^6\\ 
   &+ f_1 f_5 s_1^2 s_2^4
   + 4 f_1 f_6 s_1^3 s_2^3 + 2 f_2^2 s_2^6 + 2 f_2 f_6 s_1^4 s_2^2 
   + 2 f_4 f_6 s_1^6 + \hbox{O}(8),\\
k_2 = &2 s_1 s_2 + f_1 s_2^4 + f_3 s_1^2 s_2^2 + f_5 s_1^4 + f_0 f_3 s_2^6
   + 8 f_0 f_4 s_1 s_2^5 + 16 f_0 f_5 s_1^2 s_2^4\\ 
   &+ 40 f_0 f_6 s_1^3 s_2^3 - 2 f_1 f_2 s_2^6 + 2 f_1 f_4 s_1^2 s_2^4 
   + 6 f_1 f_5 s_1^3 s_2^3 + 16 f_1 f_6 s_1^4 s_2^2 - f_2 f_3 s_1^2 s_2^4\\ 
   &+ 2 f_2 f_5 s_1^4 s_2^2
   + 8 f_2 f_6 s_1^5 s_2 - f_3 f_4 s_1^4 s_2^2 + f_3 f_6 s_1^6 
   - 2 f_4 f_5 s_1^6 + \hbox{O}(8),\\
k_3 = &s_1^2 - f_0 s_2^4 - f_4 s_1^4 + 2 f_0 f_2 s_2^6 + 2 f_0 f_4 s_1^2 s_2^4
   + 4 f_0 f_5 s_1^3 s_2^3 + 18 f_0 f_6 s_1^4 s_2^2\\ 
   &+ f_1 f_5 s_1^4 s_2^2
   + 8 f_1 f_6 s_1^5 s_2 + 4 f_2 f_6 s_1^6 - f_3 f_5 s_1^6 + 2 f_4^2 s_1^6
   + \hbox{O}(8),\\
k_4 = &1,
\end{split}
\end{equation}
where by $\hbox{O}(8)$ we mean terms of degree at least~$8$ in~$s_1,s_2$.
\par
We note that if we perform a quadratic twist on the curve given in~\Cref{genus2general}
to give $c y^2 = \mathfrak{F}(x)$, then this induces a
linear map $(k_1,k_2,k_3,k_4) \mapsto (k_1,k_2,k_3,ck_4)$
between the Kummer surfaces.
\par
We know (see p.22 of~\cite{CasselsFlynnBook})
that addition by any point of order~$2$ gives a linear map
on the Kummer surface. Specifically, if our curve has
the form $y^2 = 
(g_2 x^2 + g_1 x + g_0)(h_4 x^4 + h_3 x^3 + h_2 x^2 + h_1 x + h_0)$,
and $r_1,r_2$ are the roots of $g_2 x^2 + g_1 x + g_0$, then
addition by $\{ (r_1,0), (r_2,0) \}$ induces on $\kgen$
the linear map
\begin{equation}\label{kummeraddorder2}
\begin{pmatrix}
k_1\\
k_2\\
k_3\\
k_4\\
\end{pmatrix}
\mapsto  
W
\begin{pmatrix}
k_1\\
k_2\\
k_3\\
k_4\\
\end{pmatrix}
\end{equation} 
where~$W$ is the matrix (reproducing (3.2.10) from~\cite{CasselsFlynnBook}):
\small
\begin{equation}\label{kummeraddorder2mtx}
\begin{pmatrix}
g_2^2 h_0 + g_0 g_2 h_2 - g_0^2 h_4 &
g_0 g_2 h_3 - g_0 g_1 h_4 &
g_1 g_2 h_3 - g_1^2 h_4 + 2 g_0 g_2 h_4 &
g_2 \\
-g_0 g_2 h_1 - g_0 g_1 h_2 + g_0^2 h_3 &
g_2^2 h_0 - g_0 g_2 h_2 + g_0^2 h_4 &
g_2^2 h_1 - g_1 g_2 h_2 - g_0 g_2 h_3 &
-g_1 \\
-g_1^2 h_0 + 2 g_0 g_2 h_0 + g_0 g_1 h_1 &
-g_1 g_2 h_0 + g_0 g_2 h_1 &
-g_2^2 h_0 + g_0 g_2 h_2 + g_0^2 h_4 &
g_0 \\
w_{41} &
w_{42} &
w_{43} &
w_{44} \\
\end{pmatrix}
\end{equation}
\normalsize
and where the entries $w_{41},w_{42},w_{43},w_{44}$ of the bottom 
row are:
\small
\begin{equation}\label{Wbottomrow}
\begin{split}
& w_{41} \colonequals -g_1 g_2^2 h_0 h_1 + g_1^2 g_2 h_0 h_2 + g_0 g_2^2 h_1^2
- 4 g_0 g_2^2 h_0 h_2 - g_0 g_1 g_2 h_1 h2 + g_0 g_1 g_2 h_0 h_3
- g_0^2 g_2 h_1 h_3,\\
& w_{42} \colonequals g_1^2 g_2 h_0 h_3 - g_1^3 h_0 h_4 - 2 g_0 g_2^2 h_0 h_3
- g_0 g_1 g_2 h_1 h_3 + 4 g_0 g_1 g_2 h_0 h_4 + g_0 g_1^2 h_1 h_4
- 2 g_0^2 g_2 h_1 h_4,\\
& w_{43} \colonequals - g_0 g_2^2 h_1 h_3 - g_0 g_1 g_2 h_2 h_3 + g_0 g_1 g_2 h_1 h_4
+ g_0 g_1^2 h_2 h_4 + g_0^2 g_2 h_3^2 - 4 g_0^2 g_2 h_2 h_4
- g_0^2 g_1 h_3 h_4,\\
& w_{44} \colonequals - g_2^2 h_0 - g_0 g_2 h_2 - g_0^2 h_4.
\end{split}
\end{equation}
\normalsize
The eigenvalues of this matrix are the square roots
of the resultant of $g_2 x^2 + g_1 x + g_0$
and $h_4 x^4 + h_3 x^3 + h_2 x^2 + h_1 x + h_0$, so this map
will be diagonalisable over~$K$ exactly when this resultant
is square in~$K$.
Furthermore, if $\{ (r_1,0), (r_2,0) \}$
and $\{ (r_3,0), (r_4,0) \}$ are points of order~$2$, 
where $r_1,r_2,r_3,r_4$ are distinct, then the corresponding matrices 
commute as affine matrices. If they have a point in common then
they anticommute.
\par
As described in Theorem~3.4.1 of~\cite{CasselsFlynnBook},
there is a $4\times 4$ matrix of 
biquadratic forms~$B_{ij}$ such that, for any points
$\mathfrak{A},\mathfrak{B}$ on the Jacobian variety,
\begin{equation}\label{biquadgen}
\Bigl( B_{ij}(k(\mathfrak{A}),k(\mathfrak{B})) \Bigr)
= \Bigl( k_i(\mathfrak{A}+\mathfrak{B}) k_j(\mathfrak{A}-\mathfrak{B})
       + k_i(\mathfrak{A}-\mathfrak{B}) k_j(\mathfrak{A}+\mathfrak{B})
  \Bigr),
\end{equation}
where $k(\mathfrak{A})$ denotes 
$(k_1(\mathfrak{A}),k_2(\mathfrak{A}),k_3(\mathfrak{A}),k_4(\mathfrak{A}))$,
the image of~$\mathfrak{A}$ in~$\kgen$, and similarly for~$k(\mathfrak{B})$.
\par
If the genus~$2$ curve has the form
\begin{equation}\label{richeloteqn}
y^2 = H_1(x) H_2(x) H_3(x),
\end{equation}
and we define the $h_{ij}$ by $H_j = H_j(x) = h_{j2} x^2 + h_{j1} x + h_{j0}$,
then the Jacobian admits a Richelot isogeny (described on p.89
of~\cite{CasselsFlynnBook}) to the Jacobian of the following
curve:
\begin{equation}\label{richelotdual}
y^2 = \Delta \bigl( H_2' H_3 - H_2 H_3' \bigr)
             \bigl( H_3' H_1 - H_3 H_1' \bigr)
             \bigl( H_1' H_2 - H_1 H_2' \bigr),
\end{equation}
where $\Delta = \hbox{det} ( h_{ij} )$. This induces a
$(2,2)$-isogeny on the associated Kummer surfaces, whose
kernel consists of the identity and the elements of order~$2$
corresponding to $H_1,H_2,H_3$. There is also a Richelot isogeny
in the other direction (the dual isogeny), and the compisition 
of these gives multiplication by~$2$.
\par
We note that all of the above identities are derivable
purely algebraically; for example, the defining equation of
the above Kummer surface can be verified directly, merely
from the fact that $y_1^2 = \mathfrak{F}(x_1)$
and $y_2^2 = \mathfrak{F}(x_2)$, where we recall here that
$\{(x_1, y_1), (x_2, y_2)\}$ is an element of the Jacobian variety.

\bigskip\bigskip\bigskip

\section{The Squared Kummer model}
\label{section3}

In this section, we shall briefly recall a model of the Kummer
surface from~\cite[Ch. 4]{CossetThesis},
which we shall call~$\ksqr$, the \emph{Squared Kummer surface}, for reasons which
will become apparent in~\Cref{section4}.
\par
We define the following quantities in terms of our parameters
$a,b,c,d \in K$.
\begin{equation}\label{ABCDdefn}
\begin{split}
&A = a^2 + b^2 + c^2 + d^2,\ 
B = a^2 + b^2 - c^2 - d^2,\\ 
&C = a^2 - b^2 + c^2 - d^2,\ 
D = a^2 - b^2 - c^2 + d^2,
\end{split}
\end{equation}
and let $\gamma = \sqrt{CD/(AB)}$. We define the Rosenhain curve
\begin{equation}\label{Crosen}
\crosen : y^2 = x(x-1)(x-\lambda)(x-\mu)(x-\nu),
\end{equation}
where
\begin{equation}\label{lambdamunu}
\lambda = \frac{a^2 c^2}{b^2 d^2},\
\mu = \frac{c^2 (1 + \gamma)}{d^2 (1 - \gamma)},\
\nu = \frac{a^2 (1 + \gamma)}{b^2 (1 - \gamma)}.
\end{equation}
We impose a non-degeneracy condition on $a,b,c,d$ 
that the roots of~$\crosen$ are distinct, so that it is
of genus~$2$. This curve is defined over $K(\gamma)$;
even when $\gamma \not\in K$, it is birationally equivalent
to its $K(\gamma)/K$-conjugate.

\begin{remark}
   In previous literature, for example
   in~\cite{ChungCostelloSmithFastUniform},
   the parameters for the squared Kummer surface are often
   constructed from $a,b,c,d$ rather than $a^2, b^2, c^2, d^2$,
   as above. However, since we shall be primarily interested 
   in the fast Kummer surface model, to be described in the next section, 
   our main aim here is just to link the two models, for which 
   it will be convenient to assume that the parameters 
   are squares in our ground field~$K$. 
\par
\end{remark}

It is noted by Chung, Costello and Smith 
in~\cite[\S 3]{ChungCostelloSmithFastUniform} 
that we can apply the following linear map to the coordinates
$k_1,k_2,k_3,k_4$ of $\kgen({\crosen})$, the general Kummer model
for Rosenhain curve, to obtain the coordinates $X,Y,Z,T$
of the squared Kummer surface, namely: 
\begin{equation}\label{mapgentosqr}
\begin{pmatrix}
X\\
Y\\
Z\\
T\\
\end{pmatrix}
=
M
\begin{pmatrix}
k_1\\
k_2\\
k_3\\
k_4\\
\end{pmatrix},
\end{equation}
where
\begin{equation}\label{matrixM}
M =
\begin{pmatrix}
a^2 \mu (\lambda + \nu) &
- a^2 \mu &
a^2 (\mu + 1) &
- a^2\\
b^2 \lambda \nu (1 + \mu) &
- b^2 \lambda \nu &
b^2 (\lambda + \nu) &
- b^2\\
c^2 \nu (\lambda + \mu) &
- c^2 \nu &
c^2 (\nu + 1) &
- c^2\\
d^2 \lambda \mu (1 + \nu) &
- d^2 \lambda \mu &
d^2 (\lambda + \mu) &
- d^2
\end{pmatrix}.
\end{equation}
We note that there is an error 
in~\cite[\S 3]{ChungCostelloSmithFastUniform},
 as some of the factors
were omitted from the entries of the above matrix, and we
have corrected that error here.
\par
After applying the above linear map, the Kummer equation
is transformed to the following model.
\begin{equation}\label{ksqr}
\begin{split}
\ksqr : &\Bigl( (X^2 + Y^2 + Z^2 + T^2) - F (X T + Y Z) \\
&\ \ - G (X Z + Y T) - H (X Y + Z T) \Bigr)^2 = 4 E^2 X Y Z T,
\end{split}
\end{equation}
where
\begin{equation}\label{EFGHdefn}
\begin{split}
E &= 
abcdABCD/\bigl((a^2 d^2 - b^2 c^2)(a^2 c^2 - b^2 d^2)(a^2 b^2 - c^2 d^2)\bigr),
\\
F &= (a^4 - b^4 - c^4 + d^4)/(a^2 d^2 - b^2 c^2),\\
G &= (a^4 - b^4 + c^4 - d^4)/(a^2 c^2 - b^2 d^2),\\
H &= (a^4 + b^4 - c^4 - d^4)/(a^2 b^2 - c^2 d^2).
\end{split}
\end{equation}
We note that
$\ksqr$ is defined over~$K$, even when~$\kgen$ is defined
over~$K(\gamma)$.
\par
If we consider the identity element and the points of order~$2$ 
on the Jacobian of~$\crosen$ given by~$\{ \infty, (0,0) \}$,
$\{ (\mu, 0), (\nu, 0) \}$ and $\{ (1,0), (\lambda, 0) \}$,
then addition by these induces on $\ksqr$ the linear maps which
take $(X,Y,Z,T)$ to, respectively
\begin{equation}\label{sqrimageorder2}
(X,Y,Z,T),\ (Y,X,T,Z),\ (Z,T,X,Y),\ (T,Z,Y,X).
\end{equation}
The change in coordinates given in~$M$ has therefore not
diagonalised these maps, but has certainly greatly simplified them. 

\bigskip\bigskip\bigskip

\section{Rederivation of the Fast Kummer Surface}
\label{section4}
In this section, we recall another elegant and efficient model
of the Kummer surface, described in~\cite{GaudryFastGenus2}:
\begin{equation}\label{Gaudryeqn}
\begin{split}
 \kgaudry :\ & X^4 + Y^4 + Z^4 + T^4 - F (X^2 T^2 + Y^2 Z^2)
    - G (X^2 Z^2 + Y^2 T^2)\\
& - H (X^2 Y^2 + Z^2 T^2) + 2 E X Y Z T = 0,
\end{split}
\end{equation}
where $A,B,C,D,E,F,G,H$ are as defined in~\Cref{ABCDdefn} 
and~\Cref{EFGHdefn}.
We shall refer to this as $\kgaudry$, the Fast model of
the Kummer surface, and it will be our main focus here.
\par
We note that there is a map $(X,Y,Z,T) \mapsto (X^2,Y^2,Z^2,T^2)$
from the Fast Kummer surface in~\Cref{Gaudryeqn} to
the Squared Kummer surface model in~\Cref{ksqr} in the previous section, 
hence justifying its name.
\par
The derivation of this model in~\cite{GaudryFastGenus2}
uses identities of theta functions. A connection is given
with the Rosenhain curve~$\crosen$ in~\Cref{Crosen},
namely that its Jacobian is $(2,2)$-isogenous to
that related to the Fast Kummer, and a rationality
assumption is given, namely that $\gamma = \sqrt{CD/(AB)}$
is in the ground field~$K$. However, it is difficult from
the literature to derive the Kummer surface equation 
for $\kgaudry$ purely algebraically
or to find an explicit map from a General Kummer; even if one
were to follow through the map from the Kummer of the
Rosenhain curve, this would be a $(2,2)$-isogeny, rather
than a linear map. Our intention in this section is to
show how to derive the Fast Kummer purely algebraically,
with a linear map from the General Kummer of a curve
which is defined over the ground field of~$a,b,c,d$.
\par
To find a plausible candidate, we note that~$\kgaudry$ 
and~$\ksqr$ are $(2,2)$-isogenous via the
map $(X,Y,Z,T) \mapsto (X^2,Y^2,Z^2,T^2)$,
and we saw in the last section
that there is a linear map between~$\ksqr$ and the General Kummer
of the Rosenhain curve.
So, a natural candidate for our desired curve is one which
is $(2,2)$-isogenous to the Rosenhain curve.
If we apply the Richelot formula in~\Cref{richelotdual}
to the Rosenhain curve~$\crosen$ in~\Cref{Crosen}
we obtain (up to quadratic twist) the following curve.
\begin{equation}\label{curveD}
\calD : y^2 = x(x-1)(x-\rho)(x-\sigma)(x-\tau),
\end{equation}
where
\begin{equation}\label{rhosigmatau}
\rho = C D/(A B),\ \sigma = (a c + b d) C/ ((a c - b d) A),\
\tau = (a c + b d) D/((a c - b d) B).
\end{equation}
There is a Richelot isogeny from~$\kgen(\crosen)$ to~$\kgen(\calD)$,
and a Richelot isogeny (the dual isogeny) from~$\kgen(\calD)$ 
to $\kgen(\crosen)$. We shall not require here the explicit
equations for these Richelot isogenies, since from now onwards
we shall work entirely with~$\kgen(\calD)$.
\par
We note that, unlike~$\crosen$, the curve~$\calD$ is 
automatically defined over the ground field of~$a,b,c,d$,
without the need for any rationality assumption.
We do, however, impose a non-degeneracy condition on~$a,b,c,d$
that $\infty,0,1,\rho,\tau,\sigma$ are distinct, so
that $\calD$ is a curve of genus~$2$.
\par
We also note that the Fast Kummer surface has several
diagonalised involutions, namely if one negates any two
of $X,Y,Z,T$. This suggests that this model
might be derivable from the General Kummer of~$\calD$
by finding a linear map which diagonalises these involutions.
Let~$E_0$ denote the identity element of
the Jacobian of~$\calD$, and note that
the following points of order~$2$ have no overlap in support:
$E_1 = \{ (\sigma,0),(\tau,0) \}$, $E_2 = \{\infty, (0,0)\}$
and $E_3 = \{(1,0),(\rho,0)\}$.
\par
Our aim is to find linear change in coordinates 
which simultaneously diagonalises
addition by $E_1$, $E_2$, $E_3$. Since $E_1 = E_2 + E_3$,
it is sufficient if we simultaneously diagonalise~$E_2$ and~$E_3$.
After using~\Cref{kummeraddorder2mtx} to find the matrix
which gives addition by~$E_2$ on the General Kummer, we see that it has
eigenvalues $(a c + b d) C D/((a c - b d) A B)$
and $-(a c + b d) C D/((a c - b d) A B)$, each with an
eigenspace of dimension~$2$. Similarly, the matrix which
gives addition by~$E_3$ on the General Kummer has
eigenvalues $4 (a^2 d^2-b^2 c^2) (a^2 b^2-c^2 d^2) C D/(A B (a c-b d))^2$
and $-4 (a^2 d^2-b^2 c^2) (a^2 b^2-c^2 d^2) C D/(A B (a c-b d))^2$,
each with an eigenspace of dimension~$2$.
These commute as affine matrices and so are simultaneously
diagonalisable. We find that there is a dimension~$1$ intersection of each
eigenspace of the first matrix and each eigenspace of the second matrix,
which provides four linearly independent vectors which are common 
eigenvectors of both matrices. We put these as the columns
of the following change of basis matrix~$P$.
\begin{equation}\label{matrixP}
P = \begin{pmatrix}
 c n_1^2 A^2 B^2 &
 -d n_1^2 A^2 B^2 &
 -a n_1^2 A^2 B^2 &
 b n_1^2 A^2 B^2 \\
  2 a n_1 A B m_1 &
  2 b n_1 A B m_2 &
  -2 c n_1 A B m_3 &
  2 d n_1 A B m_4 \\
  c n_1 n_2 A B C D &
  d n_1 n_2 A B C D &
  -a n_1 n_2 A B C D &
  -b n_1 n_2 A B C D\\
  2 a n_2 C D m_1 &
  -2 b n_2 C D m_2 &
  -2 c n_2 C D m_3 &
  -2 d n_2 C D m_4 
\end{pmatrix},
\end{equation}
where
\begin{equation}\label{m1m2m3m4}
\begin{split}
m_1 &= c^2 (a^4+b^4-c^4+d^4)-2 a^2 b^2 d^2,\
m_2 = d^2 (a^4+b^4+c^4-d^4)-2 a^2 b^2 c^2,\\
m_3 &= a^2 (a^4-b^4-c^4-d^4)+2 b^2 c^2 d^2,\
m_4 = b^2 (a^4-b^4+c^4+d^4)-2 a^2 c^2 d^2,\\
n_1 &= a c - b d,\
n_2 = a c + b d.
\end{split}
\end{equation}
We define the following map:
\begin{equation}\label{mapgentogaudry}
\begin{pmatrix}
X\\
Y\\
Z\\
T\\
\end{pmatrix}
=
P^{-1} 
\begin{pmatrix} 
k_1\\
k_2\\
k_3\\
k_4\\
\end{pmatrix}.
\end{equation} 
Direct substitution (which we have verified in the Maple file in~\cite{github})
now gives that General Kummer
equation is converted to the Fast Kummer equation, as
desired, which proves the following result. 
\begin{theorem}\label{gaudrymaptheorem}
The map in~\Cref{mapgentogaudry} is a linear map from
the General Kummer of~$\calD$ from~\Cref{curveD} to 
the Fast Kummer from~\Cref{Gaudryeqn}.
\end{theorem}
We now have a direct linear map between the General Kummer
and the Fast Kummer, which was our aim in this section,
so that we now have a direct algebraic derivation of
the Fast Kummer. It also gives a simpler approach
any time computations on the associated Jacobian variety
are required. Indeed, one can use the Jacobian variety of~$\calD$
if operations are required which are not defined on the
Kummer surface, and can help with pseudo-addition on the
Kummer surface. This approach also removes the need
for any arithmetic assumption about~$\gamma$, since all
of our varieties and maps are defined over the ground
field of~$a,b,c,d$. 
\par
The relationship between the objects discussed so far is
summarised in the following commutative diagram.
\begin{equation}\label{eq:commdiagram}
\begin{array}{ccc}
\kgen(\calD) &\overset{\hbox{\scriptsize Richelot}}
{\looongrightarrow}&\kgen(\crosen)\\
\big\downarrow P^{-1}& &\big\downarrow M\\
\kgaudry&
\overset{\mathcal{S}}
{\looongrightarrow}& \ksqr,
\end{array}
\end{equation}
where: surfaces~$\kgen$, $\ksqr$ and $\kgaudry$ are as in
Equations~$($\ref{kummerequation}$)$, $($\ref{ksqr}$)$ 
and~$($\ref{Gaudryeqn}$)$, respectively; 
curves $\crosen$, $\calD$ 
are as in Equations~$($\ref{Crosen}$)$ and~$($\ref{curveD}$)$, respectively; 
linear maps $M$, $P$ are as 
in Equations~$($\ref{matrixM}$)$ and~$($\ref{matrixP}$)$, respectively; 
and $\mathcal{S}$ denotes the squaring map 
$\mathcal{S} \colon (X,Y,Z,T) \mapsto (X^2,Y^2,Z^2,T^2)$.
\par
Additions by~$E_1,E_2,E_3$ on the Fast Kummer surface have all
been diagonalised. Explicitly, addition by~$E_1$ induces the
map $(X,Y,Z,T) \mapsto (X,Y,-Z,-T)$, addition by~$E_2$
induces the map $(X,Y,Z,T) \mapsto (X,-Y,Z,-T)$
and addition by~$E_3$ induces the map $(X,Y,Z,T) \mapsto (X,-Y,-Z,T)$.
If we let $E_0,\ldots ,E_{15}$ denote the entire $2$-torsion
subgroup, the addition by these induce the map which takes
$(X,Y,Z,T)$ to, respectively:
\begin{equation}\label{order2actionGaudry}
\begin{split}
\allowdisplaybreaks
&(X,Y,Z,T),\ (X,Y,-Z,-T),\ (X,-Y,Z,-T),\ (X,-Y,-Z,T),\\
&(Y,X,T,Z),\ (Y,X,-T,-Z),\ (Y,-X,T,-Z),\ (Y,-X,-T,Z),\\
&(Z,T,X,Y),\ (Z,T,-X,-Y),\ (Z,-T,X,-Y),\ (Z,-T,-X,Y),\\
&(T,Z,Y,X),\ (T,Z,-Y,-X),\ (T,-Z,Y,-X),\ (T,-Z,-Y,X).
\end{split}
\end{equation}
The identity element on $\kgaudry$ is~$(a,b,c,d)$. If we apply the above
maps to this, we get the complete set of $2$-torsion elements
on $\kgaudry$, namely:
\begin{equation}\label{order2elementsGaudry}
\begin{split}
&(a,b,c,d),\ (a,b,-c,-d),\ (a,-b,c,-d),\ (a,-b,-c,d),\\
&(b,a,d,c),\ (b,a,-d,-c),\ (b,-a,d,-c),\ (b,-a,-d,c),\\
&(c,d,a,b),\ (c,d,-a,-b),\ (c,-d,a,-b),\ (c,-d,-a,b),\\
&(d,c,b,a),\ (d,c,-b,-a),\ (d,-c,b,-a),\ (d,-c,-b,a),
\end{split}
\end{equation}
which are the same as the list of nodes given
by Gaudry in~\cite[\S 3.4]{GaudryFastGenus2}.
We denote the action by the two torsion point 
$E_i$ by $\sigma_i: \kgaudry \rightarrow \kgaudry$
as described by~\Cref{order2actionGaudry}.
For example, $\sigma_1 : (X,Y,Z,T) \mapsto (X, Y, -Z, -T)$,
and $\sigma_4 : (X,Y,Z,T) \mapsto (Y,X,T,Z)$.
\par
After applying the linear
change of basis~$P$, we recover the biquadratic forms given by
~\cite{qDSA}, which are much simpler version than
the biquadratic forms on the General Kummer surface, as follows.
\begin{corollary}\label{biquadformscorollary}
Let $P,Q$ be points on the Jacobian of~$\calD$, and
let $(X_P,Y_P,Z_P,T_P)$, $(X_Q,Y_Q,Z_Q,T_Q)$ be their images
on the Fast Kummer surface. Define the following
$4\times 4$ matrix of biquadratic forms. 
\begin{equation}\label{biquadgaudry}
\begin{split}
B_{11} = &(X_P^2+Y_P^2+Z_P^2+T_P^2) (X_Q^2+Y_Q^2+Z_Q^2+T_Q^2)/(4 A)\\
       &+(X_P^2+Y_P^2-Z_P^2-T_P^2) (X_Q^2+Y_Q^2-Z_Q^2-T_Q^2)/(4 B)\\
       &+(X_P^2-Y_P^2+Z_P^2-T_P^2) (X_Q^2-Y_Q^2+Z_Q^2-T_Q^2)/(4 C)\\
       &+(X_P^2-Y_P^2-Z_P^2+T_P^2) (X_Q^2-Y_Q^2-Z_Q^2+T_Q^2)/(4 D),\\
B_{22} = &(X_P^2+Y_P^2+Z_P^2+T_P^2) (X_Q^2+Y_Q^2+Z_Q^2+T_Q^2)/(4 A)\\
       &+(X_P^2+Y_P^2-Z_P^2-T_P^2) (X_Q^2+Y_Q^2-Z_Q^2-T_Q^2)/(4 B)\\
       &-(X_P^2-Y_P^2+Z_P^2-T_P^2) (X_Q^2-Y_Q^2+Z_Q^2-T_Q^2)/(4 C)\\
       &-(X_P^2-Y_P^2-Z_P^2+T_P^2) (X_Q^2-Y_Q^2-Z_Q^2+T_Q^2)/(4 D),\\
B_{33} = &(X_P^2+Y_P^2+Z_P^2+T_P^2) (X_Q^2+Y_Q^2+Z_Q^2+T_Q^2)/(4 A)\\
       &-(X_P^2+Y_P^2-Z_P^2-T_P^2) (X_Q^2+Y_Q^2-Z_Q^2-T_Q^2)/(4 B)\\
       &+(X_P^2-Y_P^2+Z_P^2-T_P^2) (X_Q^2-Y_Q^2+Z_Q^2-T_Q^2)/(4 C)\\
       &-(X_P^2-Y_P^2-Z_P^2+T_P^2) (X_Q^2-Y_Q^2-Z_Q^2+T_Q^2)/(4 D),\\
B_{44} = &(X_P^2+Y_P^2+Z_P^2+T_P^2) (X_Q^2+Y_Q^2+Z_Q^2+T_Q^2)/(4 A)\\
       &-(X_P^2+Y_P^2-Z_P^2-T_P^2) (X_Q^2+Y_Q^2-Z_Q^2-T_Q^2)/(4 B)\\
       &-(X_P^2-Y_P^2+Z_P^2-T_P^2) (X_Q^2-Y_Q^2+Z_Q^2-T_Q^2)/(4 C)\\
       &+(X_P^2-Y_P^2-Z_P^2+T_P^2) (X_Q^2-Y_Q^2-Z_Q^2+T_Q^2)/(4 D),\\
B_{12} = &4 (a b (X_P Y_P X_Q Y_Q+Z_P T_P Z_Q T_Q)
       -c d (X_P Y_P Z_Q T_Q+Z_P T_P X_Q Y_Q))
          /g_{12},\\
B_{13} = &4 (a c (X_P Z_P X_Q Z_Q+Y_P T_P Y_Q T_Q)
       -b d (X_P Z_P Y_Q T_Q+Y_P T_P X_Q Z_Q))
          /g_{13},\\
B_{14} = &4 (a d (X_P T_P X_Q T_Q+Y_P Z_P Y_Q Z_Q)
       -b c (X_P T_P Y_Q Z_Q+Y_P Z_P X_Q T_Q))
          /g_{14},\\
B_{23} = &4 (b c (X_P T_P X_Q T_Q+Y_P Z_P Y_Q Z_Q)
       -a d (X_P T_P Y_Q Z_Q+Y_P Z_P X_Q T_Q))
          /g_{23},\\
B_{24} = &4 (b d (X_P Z_P X_Q Z_Q+Y_P T_P Y_Q T_Q)
       -a c (X_P Z_P Y_Q T_Q+Y_P T_P X_Q Z_Q))
          /g_{24},\\
B_{34} = &4 (c d (X_P Y_P X_Q Y_Q+Z_P T_P Z_Q T_Q)
       -a b (Z_P T_P X_Q Y_Q+X_P Y_P Z_Q T_Q))
          /g_{34},
\end{split}
\end{equation}
where $g_{ij}$ denotes $g_i g_j - g_k g_\ell$, where
$\{ i,j,k,\ell \} = \{ 1,2,3,4 \}$ and $g_1,g_2,g_3,g_4$
denote $A,B,C,D$, respectively (for example $g_{12}$
denotes $AB-CD$). For $i > j$, we define $B_{ij} = B_{ji}$.
Then these satisfy the analogous identity on the Fast Kummer
as described in~\Cref{biquadgen} for the General Kummer.
That is to say, if we let $(\xi_1,\xi_2,\xi_3,\xi_4)$
and $(\zeta_1,\zeta_2,\zeta_3,\zeta_4)$
denote $(X_{P+Q},Y_{P+Q},Z_{P+Q},T_{P+Q})$
and $(X_{P-Q},Y_{P-Q},Z_{P-Q},T_{P-Q})$ respectively, then the above
matrix $\bigl( B_{ij} \bigr)$ is projectively equal to 
the matrix $\bigl( \xi_i \zeta_j + \zeta_i \xi_j \bigr)$.
\end{corollary}
From the biquadratic forms, we may inductively define
the division polynomials which give multiplication-by-$N$, 
as follows. We initially define
\begin{equation}\label{phi0phi1}
\begin{split}
 &\phi_X^{(0)} = a, \phi_Y^{(0)} = b, \phi_Z^{(0)} = c, \phi_T^{(0)} = d,\\
 &\phi_X^{(1)} = X, \phi_Y^{(1)} = Y, \phi_Z^{(1)} = Z, \phi_T^{(1)} = T,
\end{split}
\end{equation}
and then inductively define the following, which are
polynomials modulo the equation of the Fast Kummer surface.
\begin{equation}\label{phiinductioneven}
\begin{split}
\phi_X^{(2N)} 
  &= B_{11}\bigl( (\phi_X^{(N)},\phi_Y^{(N)},\phi_Z^{(N)},\phi_T^{(N)}),
                 (\phi_X^{(N)},\phi_Y^{(N)},\phi_Z^{(N)},\phi_T^{(N)}) \bigr)/a,\\
\phi_Y^{(2N)} 
  &= B_{22}\bigl( (\phi_X^{(N)},\phi_Y^{(N)},\phi_Z^{(N)},\phi_T^{(N)}),
            (\phi_X^{(N)},\phi_Y^{(N)},\phi_Z^{(N)},\phi_T^{(N)}) \bigr)/b,\\
\phi_Z^{(2N)} 
  &= B_{33}\bigl( (\phi_X^{(N)},\phi_Y^{(N)},\phi_Z^{(N)},\phi_T^{(N)}),
            (\phi_X^{(N)},\phi_Y^{(N)},\phi_Z^{(N)},\phi_T^{(N)}) \bigr)/c,\\
\phi_T^{(2N)} 
  &= B_{44}\bigl( (\phi_X^{(N)},\phi_Y^{(N)},\phi_Z^{(N)},\phi_T^{(N)}),
            (\phi_X^{(N)},\phi_Y^{(N)},\phi_Z^{(N)},\phi_T^{(N)}) \bigr)/d,
\end{split}
\end{equation}
and similarly
\begin{equation}\label{phiinductionodd}
\begin{split}
\phi_X^{(2N+1)} 
&= B_{11}\bigl( (\phi_X^{(N+1)},\phi_Y^{(N+1)},\phi_Z^{(N+1)},\phi_T^{(N+1)}),
          (\phi_X^{(N)},\phi_Y^{(N)},\phi_Z^{(N)},\phi_T^{(N)}) \bigr)/X,\\
\phi_Y^{(2N+1)} 
&= B_{22}\bigl( (\phi_X^{(N+1)},\phi_Y^{(N+1)},\phi_Z^{(N+1)},\phi_T^{(N+1)}),
          (\phi_X^{(N)},\phi_Y^{(N)},\phi_Z^{(N)},\phi_T^{(N)}) \bigr)/Y,\\
\phi_Z^{(2N+1)} 
&= B_{33}\bigl( (\phi_X^{(N+1)},\phi_Y^{(N+1)},\phi_Z^{(N+1)},\phi_T^{(N+1)}),
          (\phi_X^{(N)},\phi_Y^{(N)},\phi_Z^{(N)},\phi_T^{(N)}) \bigr)/Z,\\
\phi_T^{(2N+1)} 
&= B_{44}\bigl( (\phi_X^{(N+1)},\phi_Y^{(N+1)},\phi_Z^{(N+1)},\phi_T^{(N+1)}),
          (\phi_X^{(N)},\phi_Y^{(N)},\phi_Z^{(N)},\phi_T^{(N)}) \bigr)/T.
\end{split}
\end{equation}
The polynomials $\phi_X^{(N)},\phi_Y^{(N)},\phi_Z^{(N)},\phi_T^{(N)}$
are each of degree~$N^2$ in $X,Y,Z,T$, and they give the
multiplication-by-$N$ map. 
\par
For odd~$N$, we note that the $2$-torsion subgroup maps to
itself under multiplication by~$N$, and the effect
of addition by these $2$-torsion elements, described
in~\Cref{order2actionGaudry} is preserved. Furthermore,
we can see inductively that these actions are preserved
on the division polynomials not merely projectively,
but transparently as affine polynomials. That is to say,
for example, replacing $X,Y,Z,T$ with $X,Y,-Z,-T$ has
precisely the effect of replacing
$\phi_X^{(N)},\phi_Y^{(N)},\phi_Z^{(N)},\phi_T^{(N)}$
with $\phi_X^{(N)},\phi_Y^{(N)},-\phi_Z^{(N)},-\phi_T^{(N)}$,
and similarly for all of the involutions in~\Cref{order2actionGaudry}.
One consequence is that the polynomial $\phi_X^{(N)}$ only includes those
monomials which are left unchanged by taking $(X,Y,Z,T)$
to any of $(X,Y,Z,T)$, $(X,Y,-Z,-T)$, $(X,-Y,Z,-T)$, $(X,-Y,-Z,T)$.
The polynomial $\phi_Y^{(N)}$ only includes those
monomials which are left unchanged by first and second of these
and are negated by the third and fourth.
The polynomial $\phi_Z^{(N)}$ only includes those
monomials which are left unchanged by first and third of these
and are negated by the second and fourth.
The polynomial $\phi_T^{(N)}$ only includes those
monomials which are left unchanged by first and fourth of these
and are negated by the second and third.
This describes a partition of all monomials of degree~$N^2$ into four parts,
and each of the degree~$N$ division polynomials only contains
those from one part. We note also that the equation of the Fast Kummer surface
itself involves only degree~$4$ monomials of the first type
(left unchanged by all four of these involutions).
\par
Addition by other points of order~$2$ gives further structure.
For example, swapping $X \leftrightarrow Y, Z \leftrightarrow T$
will have the effect 
$\phi_X^{(N)} \leftrightarrow \phi_Y^{(N)}$,
$\phi_Z^{(N)} \leftrightarrow \phi_T^{(N)}$.
\par
We also note that the construction in this section, in which
two elements of order~$2$ have their associated linear maps
simultaneously diagonlised in order to simplify the Kummer surface
(and its associated structures) is available whenever we
start with a curve of genus~$2$ of the form $y^2 = H_1(x)H_2(x)H_3(x)$,
provided that the resultant of~$H_1$ and~$H_2 H_3$,
and the resultant of~$H_2$ and~$H_1 H_3$ are squares in~$K$
(it follows that the resultant of $H_3$ and $H_1 H_2$ will
also be square in~$K$). This ensures that we can simultaneously
diagonalise that maps for addition by the points of order~$2$
corresponding to $H_1,H_2,H_3$ since they will commute as
affine matrices and their eigenvalues will be in~$K$.
The result will not necessarily be as elegant 
as the Fast Kummer surface model 
(for example, it will not typically be monic in all
of $X^4,Y^4,Z^4,T^4$),
but it will be just as sparse, consisting
only of terms involving the monimials
$X^4$, $Y^4$, $Z^4$, $T^4$, $X^2 Y^2$, $Z^2 T^2$,
$X^2 Z^2$, $Y^2 T^2$, $Z^2 T^2$, $Y^2 Z^2$,
and the associated biquadratic forms will also involve the
same monomials as for the Fast Kummer model. 
Furthermore, the resulting model would be applicable to
a wider set of Kummer surfaces. We demonstrate this in the following example.

\begin{example}
   Consider the genus-$2$ hyperelliptic curve 
   $C \colon y^2 = H_1(x) H_2(x) H_3(x)$ defined over $\F_{101}$ with 
   \begin{align*}
      H_1(x) &= x^2 + 15x + 13, \\
      H_2(x) &= x^2 + 53x + 83,\\
      H_3(x) &= x^2 + 10x + 64.
   \end{align*} 
   and corresponding General Kummer surface 
   \begin{align*}
      \kgen \colon 84k_1^4 &+  k_1^3k_2 + 3k_1^3k_3 + 11k_1^3k_4 + 7k_1^2k_2^2 + 93k_1^2k_2k_3 + 33k_1^2k_2k_4 + 
   64k_1^2k_3^2 \\ &+ 96k_1^2k_3k_4 + 50k_1k_2^3 + 76k_1k_2^2k_3 + 49k_1k_2k_3^2 + 
   9k_1k_2k_3k_4  
   96k_1k_3^3 \\ &+ 25k_1k_3^2k_4 + 97k_1k_3k_4^2 + 11k_2^4 + 66k_2^3k_3 + 
   96k_2^2k_3^2 + k_2^2k_4^2 + 18k_2k_3^3 \\ &+ 46k_2k_3^2k_4 + 49k_3^4 + 97k_3^3k_4.
   \end{align*}
   None of $H_1, H_2, H_3$ have roots in $\F_{101}$, and therefore $C$ cannot be put into Rosenhain form. As a result, $\kgen$ cannot be put into Fast Kummer form. However, the resultants ${\rm{res}}(H_1, H_2H_3)$ and ${\rm{res}}(H_2, H_1H_3)$ are squares in $\F_{101}$. Indeed, ${\rm{res}}(H_1, H_2H_3) = 78 = 52^2$ and 
   ${\rm{res}}(H_2, H_1H_3) = 79 = 68^2$. 
   On $\kgen$, addition by the $2$-torsion points corresponding to the factor $H_1$ is given by the matrix 
   \[ M_1 = \begin{pmatrix} 6 & 18 & 39 & 1 \\ 
      92 & 13 & 15 & 86 \\
      17 & 52 & 22 & 13 \\
      46 & 54 & 52 & 60 \end{pmatrix},\]
   with eigenvalues $52$ and $49$.
   For the $2$-torsion points corresponding to $H_2$ addition is given by 
   \[ M_2 = \begin{pmatrix} 58 & 100 & 96 & 1 \\ 
      60 & 91 & 13 & 48 \\
      32 & 15 & 52 & 83 \\
      66 & 22 & 58 & 1 \end{pmatrix},\]
   with eigenvalues $33$ and $68$.
   The matrices $M_1$ and $M_2$ commute and so can be simultaneously diagonalised using the matrix 
   \[ P = \begin{pmatrix} 1&  37 & 4 & 88 \\ 
      1 & 13 & 6 & 65 \\
      1 & 76 & 66 & 99 \\
      1 & 12 & 30 & 91 \end{pmatrix}, \]
   which we find by following the procedure described in the paragraphs before~\Cref{matrixP}.
   Setting 
   \begin{equation*}
      \begin{pmatrix}
      X\\
      Y\\
      Z\\
      T\\
      \end{pmatrix}
      =
      P^{-1} 
      \begin{pmatrix} 
      k_1\\
      k_2\\
      k_3\\
      k_4\\
      \end{pmatrix},
      \end{equation*}
      the General Kummer surface $\kgen$ is transformed into a Kummer surface $\mathcal{K}$ given by equation:
      \begin{align*}
         \mathcal{K} \colon 50X^4 &+ 57Y^4 + 27T^4 + 83Z^4 + 70X^2Y^2 + 10Z^2T^2 \\
         & + 54X^2Z^2 +  91Y^2T^2 + 13X^2T^2  + 44Y^2Z^2 + 90XYZT.
      \end{align*} 
      We see that $\mathcal{K}$ has a sparse defining equation, similar to the Fast Kummer model.
\end{example}

\bigskip\bigskip

\section{Computation of~\texorpdfstring{$(N,N)$}{(N,N)}-isogenies 
between Kummer surfaces}
\label{section5}

In this section, we give algorithms to compute $(N, N)$-isogenies,
for $N$ odd,
between various Kummer surface models. For our methods
we assume the biquadratic forms $B_{i,j}$ 
(for $1 \leq i,j \leq 4$) associated to 
our domain Kummer surface $\K$ are known and efficiently
computable.

Let $\J$ be a Jacobian of a hyperelliptic curve of genus~$2$ curve with corresponding Kummer 
surface $\K$.
Let $R, S \in \J[N]$ be points of order~$N$ on $\J$ with images 
$k(R), k(S)$ on $\K$, respectively.
If $\langle R,S\rangle \subset \J[N]$ is a maximal isotropic subgroup,
$\langle R,S\rangle$ is the kernel of an $(N, N)$-isogeny 
$\Phi: \J \rightarrow \J' = \J/\langle R, S \rangle$ between Jacobians. 
This $(N,N)$-isogeny descends to a morphism of Kummer surfaces 
\begin{align*}
   \varphi: \K &\rightarrow \K' \\
   (k_1,k_2,k_3,k_4) &\mapsto (k'_1, k'_2, 
   k'_3, k'_4).
\end{align*} 
By abuse of terminology, we say $\varphi$ is an 
$(N,N)$-isogeny between Kummer surfaces with kernel generated 
by $N$-torsion points $R, S \in \K[N]$. We denote the coordinates 
of $R$ by $(R_1, R_2, R_3, R_4)$, and similarly for $S$. 
\par
For the rest of the article, our aim is to compute this $(N,N)$-isogeny, 
where the Kummer surfaces are in a desired model.
This can be separated into three main algorithms: 
\begin{itemize}
    \item Computing a basis for the space of degree $N$ homogeneous 
    forms that are invariant under addition by the $N$-torsion point $R$. 
    We repeat this for the other kernel generator $S$, to obtain two 
    bases $B_R$ and $B_S$. These
    are constructed using the biquadratic forms associated to the 
    Kummer surface.
    \item Computing a basis for the intersection of the two spaces 
    generated by $B_R$ and $B_S$. The intersection will have expected 
    dimension $4$, and a basis for this space will give the
    coordinates of a map $\psi :
    \K \rightarrow \widetilde{\K}$, which has kernel $\langle R, S
    \rangle$, but $\widetilde{\K}$ is not necessarily in desired form.
    \item Compute a linear map $\lambda$ that brings $\widetilde{\K}$ 
    into desired form $\K'$. 
\end{itemize}
We remark that the first two algorithms are applicable to any
model of Kummer surface with efficiently computable biquadratic
forms. In this article, we show how to find the final scaling 
for the General and Fast Kummer surfaces.
\par
\noindent\emph{Notation.} We will give precise operation counts for our 
algorithms, and denote a multiplication, inversion, squaring,
computing square roots, 
and addition in $K$ by {\tt M}, {\tt I}, {\tt S}, {\tt Sq} and {\tt a}.

\vspace{1ex}

\noindent {\bf Step 1: Find forms invariant under translation 
by an $N$-torsion point.} Use the biquadratic forms $B_{ij}$
associated to $\K$ to define the forms in $k_1,k_2,k_3,k_4$
of degree~$N$ which are invariant under addition by~$R$,
and the forms of degree~$N$ which are invariant under addition 
by~$S$.
\par
For general odd $N = 2n+1$, let $P \in \K$ be a point
and define
\begin{equation}
\begin{split}
R^{(\ell)}_{ij}(P ) 
&= B_{ij}( P, \ell R ),
\hbox{ for }\ell\in\{1,\ldots,n\}, \\
S^{(\ell)}_{ij}( P ) 
&= B_{ij}( P, \ell S ),
\hbox{ for }\ell\in\{1,\ldots,n\}.
\end{split}
\end{equation}
Then, for $i_1, \dots, i_N \in \{1,\dots, 4\}$, we take 
\begin{equation}\label{RSgeneralN}
   F_{R, N}\bigl((i_1,\dots,i_N), P\bigr) = \sum k_{i_1}(P)R^{(1)}_{i_2 i_3}(P) 
R^{(2)}_{i_4 i_5}(P) \cdots R^{(n)}_{i_{N-1} i_N}(P),
\end{equation}
where the sum is taken over all permutations of
$i_1,\ldots,i_N$. 
\begin{lemma}\label{lemma:Rij-invariant}
    Let $(i_1, \dots, i_N) \in \{1, \dots, 4\}^N$.
    The homogeneous form $F_{R, N}\bigl((i_1,\dots,i_N), P\bigr)$ is of degree~$N$
    and is invariant under translation-by-$R$. 
\end{lemma}
\begin{proof}
   Firstly, each $R^{(\ell)}_{i,j}$ is a homogeneous form of degree~$2$, and so 
   $F_{R, N}\bigl((i_1,\dots,i_N), P\bigr)$ is of degree $2\left ( \frac{N-1}{2} \right ) + 1 = N$.
   We now show that these forms are invariant under translation-by-$R$.
   Letting $I \colonequals (i_1, \dots, i_N) \in \{1, \dots, 4\}^N$, we define the set $\mathcal{S}_I$ as
   \[ \mathcal{S}_I \colonequals \left \{ \sigma \bigl((i_1, i_2, i_4, \dots, i_{N-1}, 
   i_N, \dots, i_5, i_3)\bigr) \colon \sigma \in \left\langle {\rm{id}}, (2 \, 3), 
   \dots, (N-1 \, N)\right\rangle  \right \}. \] 
   We have 
   \[k_{i_1}R^{(1)}_{i_2 i_3} 
   R^{(2)}_{i_4 i_5} \cdots R^{(n)}_{i_{N-1} i_N} = 
   \sum_{(j_1, \dots, j_N) \in \mathcal{S}_I} k_{j_1}(P)k_{j_2}(P+R)\cdots k_{j_N}(P+(N-1)R). \] 
   Defining \[ q_{j_1, \dots, j_N}(P) \colonequals k_{j_1}(P)k_{j_2}(P+R)\cdots k_{j_N}(P+(N-1)R),\] 
   we verify that 
   \begin{align*}
      q_{j_1, \dots, j_N}(P+R) &= k_{j_N}(P)k_{j_1}(P+R)\cdots k_{j_{N-1}}(P+(N-1)R) \\
                              &= q_{j_{N}, j_1, \dots, j_{N-1}}(P).
   \end{align*}
   Therefore, we find that
   \begin{align*}
      F_{R, N}\bigl(I, P+R \bigr) &= 
   \sum_{I} \sum_{(j_1, \dots, j_N) \in \mathcal{S}_I} q_{j_1, \dots, j_N}(P+R) \\
   & = \sum_{I} \sum_{(j_1, \dots, j_N) \in \mathcal{S}_I} q_{j_N, j_1, \dots, j_{N-1}}(P) \\ 
   & = \sum_{I} \sum_{(j_1, \dots, j_N) \in \mathcal{S}_{\widetilde{I}}} q_{j_1, \dots, j_N}(P) \\ 
   \end{align*}  
   where $\widetilde{I} = (i_N, i_1, \dots, i_{N-1})$ for $I = (i_1, \dots, i_{N})$. Then 
   \begin{align*}
      F_{R, N}\bigl(I, P+R \bigr) & = \sum_{I} k_{i_N}R^{(1)}_{i_1 i_2} \cdots R^{(n)}_{i_{N-2} i_{N-1}} 
      \\ &= \sum_{I} k_{i_1}R^{(1)}_{i_2 i_3} \cdots R^{(n)}_{i_{N-1} i_N} \\
   &= F_{R, N}\bigl(I, P \bigr),
   \end{align*}  
   as required.
\end{proof}
\begin{remark}
   For ease of notation, instead of $F_{R, N}\bigl((i_1,\dots,i_N), P\bigr)$ we will 
   write $F_{R, N}(i_1,\dots,i_N)$ when the point $P \in \K$ 
   is implicit. 
\end{remark}
As an optimisation, we may restrict that $i < j$
within each $R^{(\ell)}_{ij}$, and count this form several times in the sum 
(rather than recomputing it multiple times). 
By~\Cref{lemma:Rij-invariant}, $F_{R, N}(i_1,\ldots,i_N)$
will generate a space $\spaceR$ of homogeneous degree~$N$
forms that are invariant under translation by $R$.
We similarly define the space $\spaceS$ from the 
homogeneous forms $F_{S, N}(i_1,\dots,i_N)$.
\par
In the next proposition, we show that these spaces have 
dimension $2(N+1)$.

\begin{proposition}\label{prop:dim-and-basis}
   Let $R \in \K[N]$ be a point of order~$N$ on $\K$.
   Let $\spaceR$ be the space generated by homogeneous forms of degree~$N$
   on $\K$ that are invariant under translation by $R$. 
   Then $\dim(\spaceR) = 2(N+1)$. 
\end{proposition}
\begin{proof}
   To prove this proposition we use the theory of theta functions.
   Let $\J$ be a Jacobian of genus~$2$ curve corresponding to the Kummer 
   surface $\K$.
   Let $\Lk{}$ be the symmetric line bundle giving rise to the principal polarisation on $\J$.

   The theta functions $\{ \theta_{\beta}\}_{\beta \in (\Z/2N\Z)^2}$, 
   as defined in, for example,~\cite[Definition 2.1]{FlynnMakdisi} 
   (taking $k = 2N$), form a basis for $H^0(\J, \Lk{2N})$. Note here $\beta$ lies in $(\Z/2N\Z)^2$ by its identification with a maximal isotropic subgroup of $\J[2N]$.

   The action of $[-1]^*$ 
   on the global sections $\{ \theta_\beta \}_{\beta \in (\Z/2N\Z)^2}$ is described by 
   the general 
   transformation formula $\theta_{\beta}(-z) = \theta_{-\beta}(z)$. 
   If $\theta_{\beta}$ is fixed under this action then $\beta$ is a $2$-torsion point in $(\Z/2N\Z)^2$, i.e., $\beta \in \{ (0,0), (N,0), (0,N), (N,N) \}$. Furthermore, a sum of these global sections is invariant under this action if and only if it is of the form 
   \[ \sum_{{\rm{ord}}(\beta) \neq 2} \bigl[ c_\beta(\theta_{\beta} + \theta_{-\beta}) \bigr] \, + c_1\theta_{(0,0)} +c_2\theta_{(N,0)}+c_3\theta_{(0,N)}+c_4\theta_{(N,N)},\]
   for some $c_\beta, c_1, c_2, c_3, c_4 \in \Z$. 

   The space of degree~$N$
   homogeneous forms on $\K$ is
   generated by the subspace of $H^0(\J, \Lk{2N})$
   that is invariant under the action of $[-1]^*$, and, by the discussion above, a basis for this space is given by
   \begin{align*}
      \Big\{ \theta_{(0,0)}, \theta_{(N,0)}, 
   \theta_{(0, N)}, &\theta_{(N, N)}\Big\} \ \cup \  
   \Bigg\{ \theta_{(b_1, b_2)} + \theta_{(-b_1, -b_2)} \ \Bigg| \ \scriptstyle{(b_1, b_2) \,
   \neq} {\begin{array}{l} \scriptstyle{(0, \, 0), \, (N, \, N),} \\ 
         \scriptstyle{(0, \, N), \, (N, \, 0)} \end{array}} \Bigg\}
   \end{align*}
   The size of this basis is $4 + \frac{1}{2}(4 N^2 - 4) = 2(N^2 + 1)$, and therefore 
   the dimension of the space of degree~$N$
   homogeneous forms on $\K$ is $ 2(N^2 + 1)$. 

   Given this basis, we now find the forms that generate $\spaceR$, 
   i.e., the space of degree~$N$
   homogeneous forms on $\K$ that are invariant under translation by an $N$-torsion point $R$. 

   Using the identity,  
   for a suitable primitive $2N$-th root of unity $\zeta_{2N}$, for an $N$-torsion point $r = (r_1, r_2) \in (\Z/2N\Z)^2$ we have
   \[ \theta_{(b_1, b_2)}(z + r) = (\zeta_{2N})^{r_1b_1 + r_2b_2}\theta_{(b_1, b_2)}(z).\]
   We immediately see that 
   \[ \theta_{(0,N)}(z + r) = (\zeta_{2N})^{N \cdot r_2}\theta_{(b_1, b_2)}(z) = \theta_{(b_1, b_2)}(z),\]
   as $r_2$ is an $N$-torsion point in $\Z/2N\Z$.
   This also holds for $\theta_{(N,0)}$, $\theta_{(0,0)}$, and $\theta_{(N,0)}$.

   Note that $(r_1,r_2) \in \{(2,0), (0,2), (2,2)\}$. Take, for example $r = (2,0)$ and
   consider basis elements of the form $(\theta_{(b_1, b_2)} + \theta_{(-b_1, -b_2)})(z)$. 
   By indepedence of the $\theta_\beta$, it is invariant under translation-by-$r$ if and only if 
   \[ (\zeta_{2N})^{2b_1} = 1 \iff 2b_1 \equiv 0 \bmod 2N \iff b_1 = 0 \text{ or } N.\]
   Therefore, a basis $\basisR$ for the space $\spaceR$ is given by 
   \begin{align*}
      \Big\{ \theta_{(0,0)}, \theta_{(N,0)}, 
   \theta_{(0, N)}, &\theta_{(N, N)}\Big\} \ \cup \  
   \Bigg\{ \theta_{(b_1, b_2)} + \theta_{(-b_1, -b_2)} \ \Bigg| \ 
   \begin{array}{l} b_1 = 0, N \text{ and } b_2 \neq 0, N  \\ 
      \text{up to } (b_1, b_2) \mapsto (b_1, -b_2) \end{array} \Bigg\}
   \end{align*}
   For basis elements 
   of the form $\theta_{(b_1, b_2)} + \theta_{(-b_1, -b_2)}$, we have 
   $2$ choices for $b_1$ and $(2N-1)/2 = N-1$ choices for $b_2$. In particular, 
   the basis has size $\dim(\spaceR) = 4 + 2(N-1) = 2(N+1)$. The same holds when $r = (0,2)$ or $(2,2)$.
\end{proof}

\begin{conjecture}\label{conj:basis-space}
   A basis for the space $\spaceR$, defined as in \Cref{prop:dim-and-basis} 
   is given by the forms described in~\Cref{RSgeneralN} corresponding to indices 
   $(i_1, \dots,  i_N) \in \mathcal{I}_N$ where
   \begin{align*}
      \mathcal{I}_N:= \big\{ & \{1,1,\dots,1,1\}, \{1,1,\dots,1,2\}, 
      \dots, \{1,2,\dots,2,2\}, \{2,2,\dots,2,2\}, \\
         & \{3,3,\dots,3,3\}, \{3,3,\dots,3,4\}, \dots, 
         \{3,4,\dots,4,4\}, \{4,4,\dots,4,4\} \big\}.
   \end{align*}
\end{conjecture}
For General and Fast Kummer surfaces we have confirmed this experimentally for odd $N \leq 19$,
and we therefore develop our algorithms under this assumption. 
We highlight, however, that if this is not the case for larger $N$ or other 
Kummer surface models, one can compute 
the invariant forms corresponding to all possible indices $(i_1, \dots, i_N) \in \{1,2,3,4\}^N$ 
and then use linear algebra to find a basis for the space $\spaceR$. The indices corresponding 
to a basis should be independent of the $N$-torsion point $R$, and this can therefore be computed 
once and used for all subsequent computations.

\begin{remark}
   There are other choices of bases, such as the basis given by forms corresponding to the 
   following indices:
   \begin{align*}
      \big\{ & \{1,1,\dots,1,1\}, \{1,1,\dots,1,4\}, 
         \dots, \{1,4,\dots,4,4\}, \{4,4,\dots,4,4\}, \\
            & \{2,2,\dots,2,2\}, \{2,2,\dots,2,3\}, \dots, 
            \{2,3,\dots,3,3\}, \{3,3,\dots,3,3\} \big\}.
   \end{align*}
However, we found no computational advantage in choosing another basis. If future research 
reveals that another choice results in a better basis, the algorithms presented in this 
article can be easily adapted.
\end{remark}

By~\Cref{prop:dim-and-basis} and assuming~\Cref{conj:basis-space}, given above, we obtain a simple 
method for generating a basis of the spaces $\basisR$ and $\basisS$, where
$R,S$ are $N$-torsion points generating the kernel of the $(N,N)$-isogeny, 
which we give in~\Cref{alg:findbasis}. In Line~\ref{line:findbasis:sum}
of~\Cref{alg:findbasis}, the sum is taken over all permutations of
$I$, as described in~\Cref{RSgeneralN}. 
For our optimised implementation attached to this article given in~\cite{github}, we note that 
$R^{(\ell)}_{i,j} = R^{(\ell)}_{j,i}$ and avoid superfluous computations by computing the 
summand $k_{i_1}R^{(1)}_{i_2 i_3} 
R^{(2)}_{i_4 i_5} \ldots R^{(n)}_{i_{N-1} i_N}$ only for $(i_1, \dots, i_N)$ such 
that $i_j \leq i_{j+1}$ for $j > 1$ and $j \equiv 0 \bmod 2$, and scalar multiply by 
how many such summands occur in the sum for $F_{R,N}(I)$. 

\begin{algorithm}
	\caption{\textsf{FindBasis}: find a basis for $\spaceR$, $R$ 
   is a point of odd order $N$ on $\kgaudry_{a,b,c,d}$}
   \label{alg:findbasis}
	\begin{algorithmic}[1]
        \REQUIRE{An $N$-torsion point $R$ on the Kummer surface $\K$ with
        coordinates $(k_1, k_2, k_3, k_4)$, multiples $(2R, 3R, \dots, ((N-1)/2)R)$
        of the point $R$,
        and the list of indices $\mathcal{I}_N$.}
        \ENSURE{A basis $B_R$ of size $2(N+1)$ generating $\spaceR$.}
         \STATE $\basisR = \{ \}$
         \STATE $n = (N-1)/2$
         \FOR{$\ell$ from $1$ to $n$}
         \STATE Compute $R_{ij}^{(\ell)}(k_1, k_2, k_3, k_4) = 
         B_{ij}((k_1, k_2, k_3, k_4), \ell R)$ for all $i,j \in \{1,2,3,4\}$\label{line:findbasis:biquads}
         \ENDFOR
        \FOR{$I$ in $\mathcal{I}_N$}
            \STATE $F_{R, N}(I) = \sum k_{i_1}R^{(1)}_{i_2 i_3} 
            R^{(2)}_{i_4 i_5} \ldots R^{(n)}_{i_{N-1} i_N}$\label{line:findbasis:sum}
            \STATE Add $F_{R, N}(I)$ to $\basisR$
        \ENDFOR
        \RETURN $\basisR$
    \end{algorithmic} 
\end{algorithm}

\begin{proposition}\label{prop:findbasis}
   The cost of finding the basis $\basisR$ using~\Cref{alg:findbasis} with $\mathcal{I}_N$ as in~\Cref{conj:basis-space} is bounded by 
   $$\cost{basis} \leq 9(3^{(N-1)/2} - 1) \, {\tt M}_{\text{poly}} + 2(3^{(N-1)/2} - N - 1) 
   \, {\tt a}_{\text{poly}} + \frac{N-1}{2} \cdot \cost{biquad},$$
   where ${\tt M}_{\text{poly}}$ and ${\tt a}_{\text{poly}}$ represents a multiplication 
   and an addition in $K[k_1,k_2,k_3,k_4]$ (respectively) and $\cost{biquad}$ is the cost of 
   computing the evaluated biquadratic form $R^{(\ell)}_{ij} = B_{ij}((k_1,k_2,k_3,k_4), \ell R)$ 
   for some $1 \leq \ell \leq (N-1)/2$. 
\end{proposition}
\begin{proof}
   To compute the $F_{R,N}(I)$ for all $$I \in \big\{ \{1,1,\dots,1,1\}, \{1,1,\dots,1,2\}, 
   \dots, \{1,2,\dots,2,2\}, \{2,2,\dots,2,2\} \big\},$$ 
   we first compute $k_{i_1}R^{(1)}_{i_2 i_3} 
   R^{(2)}_{i_4 i_5} \ldots R^{(n)}_{i_{N-1} i_N}$ for all $(i_1, \dots, i_N)$ with 
   $i_k \in \{1, 2\}$ for all $k = 1, \dots, N$ and $i_j \leq i_{j+1}$ for $j > 1$ and $j \equiv 0 \bmod 2$. 
   
   We start with the forms $R^{(1)}_{1, 1}, R^{(1)}_{1, 2}, R^{(1)}_{2,2}$. 
   The first step is to multiply each of these forms by $R^{(2)}_{1, 1}, R^{(2)}_{1, 2}$, 
   and $R^{(2)}_{2,2}$ to obtain 
   \begin{align*}
      R^{(1)}_{1, 1}R^{(2)}_{1, 1}, \ R^{(1)}_{1, 1}R^{(2)}_{1, 2}, \ R^{(1)}_{1, 1}R^{(2)}_{2,2}, \\
      R^{(1)}_{1, 2}R^{(2)}_{1, 1}, \ R^{(1)}_{1, 2}R^{(2)}_{1, 2}, \ R^{(1)}_{1, 2}R^{(2)}_{2,2}, \\
      R^{(1)}_{2,2}R^{(2)}_{1, 1}, \ R^{(1)}_{2,2}R^{(2)}_{1, 2}, \ R^{(1)}_{2,2}R^{(2)}_{2,2}.
   \end{align*}
   At step $m$ we multiply the output of the previous step with $R^{(k+1)}_{1, 1}, R^{(k+1)}_{1, 2}$, 
   and $R^{(k+1)}_{2,2}$. We continue this until $m = (N-1)/2 - 1$. 
   Each step requires $3^{m+1}$ multiplications in $K[k_1,k_2,k_3,k_4]$, so to run all steps we need at 
   most $\sum_{m = 1}^{(N-1)/2} 3^{m+1} = \frac{9}{2}(3^{(N-1)/2} - 1)$ multiplications in $K[k_1,k_2,k_3,k_4]$. 
   Then, we multiply each of these forms by $k_1$ and then by $k_2$ to obtain all forms needed to compute all 
   the $F_{R,N}(I)$. Note that representing the homogeneous forms as an array of coefficients, this final step 
   does not require any multiplications in $K$.
   We then construct all the $F_{R,N}(I)$ from these forms using $3^{(N-1)/2} - (N+1)$ additions in $K[k_1,k_2,k_3,k_4]$.
   \par
   The same argument holds for indices in
   $$\big\{ \{3,3,\dots,3,3\}, \{3,3,\dots,3,4\}, \dots, 
   \{3,4,\dots,4,4\}, \{4,4,\dots,4,4\} \big\},$$
   and we obtain a cost of $9(3^{(N-1)/2} - 1)$ multiplications and $2\cdot 3^{(N-1)/2} - 2(N+1)$ 
   additions in $K[k_1,k_2,k_3,k_4]$ to compute the basis for $\spaceR$ using \textsf{FindBasis} 
   given biquadratics $R^{(\ell)}_{ij}$ for
   $1 \leq \ell \leq (N-1)/2$.

   The cost of computing $R^{(\ell)}_{ij}$ for all
   $1 \leq \ell \leq (N-1)/2$ is $\frac{N-1}{2} \cdot \cost{biquad}$. Combining these costs, 
   we get the upper bound given in 
   the statement of the proposition.
\end{proof}

\begin{remark}
    In~\Cref{prop:findbasis}, we do not give a precise cost for ${\tt M}_{\text{poly}}$ and 
    $\cost{biquad}$ in terms of $K$-operations. Indeed, ${\tt M}_{\text{poly}}$ is hard to estimate as the size of the input is not
    fixed along the computation.
    Furthermore, $\cost{biquad}$ 
    depends on the model of the Kummer surface we are working with. 
    In~\Cref{subsubsection:fullalgorithm-fast}, 
    we give a precise upper bound for 
    the cost of $\cost{biquad}$ for Fast Kummer surfaces.
\end{remark}

\noindent {\bf Step 2: Compute the intersection.} 
Once we have a bases $B_R$ and $B_S$ for the spaces $\spaceR$ 
and $\spaceS$, 
we compute the intersection $\spaceRS$ of expected dimension $4$, 
with basis $\basisRS = \{ \psi_1, \psi_2, \psi_3, \psi_4 \}$.
These basis elements will give the coordinates of the degree-$N$ map 
$\psi$. 
To find the $\psi_i$, we consider the equation
\begin{equation}\label{eq:intersection}
   \sum_{f \in \basisR} c_f \cdot f - \sum_{g \in \basisS} c_g \cdot g = 0,
\end{equation}
where the coefficients $c_f$ and $c_g \in K$. 
As this equation must hold for any point $(k_1, k_2, k_3, k_4)$
on $\K$, the coefficients of each monomial in
~\Cref{eq:intersection} must be identically zero. This gives us 
a linear system over the ground field $K$ in the coefficients 
$c_f, c_g \in \K$. To find $\psi$, we compute
the kernel basis vectors
\begin{equation}\label{eq:kernelforintersection}
    \begin{split}
    \Bigg\{ \begin{array}{c} (c_{f, 1} : f \in \basisR, \ c_{g, 1} : g \in \basisS), 
(c_{f, 2} : f \in \basisR, \ c_{g, 2} : g \in \basisS), \\
(c_{f, 3} : f \in \basisR, \ c_{g, 3} : g \in \basisS), 
(c_{f, 4} : f \in \basisR, \ c_{g, 4} : g \in \basisS) \end{array} \Bigg\},
    \end{split}
\end{equation}
of the matrix corresponding to this linear system.
Then each $\psi_i$ is given by the sum 
$\psi_i = \sum_{f \in B_R} c_{f, i} \cdot f.$
We summarise this in~\Cref{alg:findintersection}.
In Line~\ref{line:findintersection:basisforkernel} of~\Cref{alg:findintersection}, on input 
a matrix $M$, \textsf{BasisForKernel}($M$)
will compute a basis for the kernel of that matrix. 

\begin{algorithm}
	\caption{\textsf{FindIntersection}: find a basis for the 
    intersection $\spaceRS$}
   \label{alg:findintersection}
	\begin{algorithmic}[1]
        \REQUIRE{A basis $\basisR$ for the space $\spaceR$ and a 
        basis $\basisS$ for the space $\spaceS$}
        \ENSURE{The generator of the intersection $\spaceRS
         = \spaceR \cap \spaceS$}
         \STATE Let ${\tt mons}$ be the monomials in $\basisR$ 
         and $\basisS$
         \STATE $\basisR = \{f_1, \dots, f_{2(N+1)} \}$
         \STATE $\basisS = \{g_1, \dots, g_{2(N+1)} \}$
         \FOR{$j = 1$ to $4(N+1)$}
            \FOR{$i = 1$ to $2(N+1)$}
               \STATE Let $c_{i,j}$ be the coefficient of ${\tt mons}_j$ in $f_i$
            \ENDFOR
            \FOR{$i = 2(N+1) + 1$ to $4(N+1)$}
               \STATE Let $c_{i,j}$ be the coefficient of ${\tt mons}_j$ in $-g_i$
            \ENDFOR
         \ENDFOR
         \STATE $M = (c_{i,j})_{i = 1, \dots, 4(N+1), j = 1, \dots, 2(N+1)}$
         \STATE $(d_1, \dots, d_{2(N+1)}) = \textsf{BasisForKernel}(M)$
         \label{line:findintersection:basisforkernel}
         \STATE $F = \sum_{i = 1}^{2(N+1)} d_if_i$
        \RETURN $F$
    \end{algorithmic} 
\end{algorithm}

\begin{proposition}\label{prop:findintersection}
   The cost of computing the intersection $\spaceRS
   = \spaceR \cap \spaceS$ with~\Cref{alg:findintersection} is 
   \begin{align*}
      \cost{$\cap$} \leq \, \frac{2}{3}(N + 1)(2 N^3 + 44 N^2 + 122 N + 69) \,  {\tt M} + 
    \frac{2}{3}(8N + 5)(4N^2 + 11N + 9) \ {\tt a}.
   \end{align*}
\end{proposition}
\begin{proof}
   The linear system of equations given by 
   $$\sum_{f \in \basisR} c_f \cdot f - \sum_{g \in \basisS} 
   c_g \cdot g$$
   is a system of $m$ equations in $\#\basisR + \#\basisS = 
   4(N+1)$ unknowns, where 
   $m$ is the number of monomials in $\basisR$ and $\basisS$. 
   For $R,S$ generating the kernel of an $(N,N)$-isogeny, the dimension 
   of the solution space will be $4(N+1)$, thus it suffices to consider $4(N+1)$ of the $m$ equations.
   In this way, we can obtain the coefficients $c_f, c_g \in K$
   using Gaussian elimitation. 
   Since we work projectively, we may adapt Farebrother~\cite{Farebrother} to remove inversions 
   (at the cost of more multiplications), we find that we find the coefficients with $\frac{64}{3}N^3 + 72N^2 + \frac{230}{3}N + 26$
   multiplications, 
   and $\frac{64}{3}N^3 + 88N^2 + \frac{314}{3}N + 38$
   additions in $K$.
   
   Recalling that $\#\basisR = 2(N+1)$, 
   constructing the basis of the intersection via 
   $\psi_i = \sum_{f \in B_R} c_{f, i} f$ requires 
   at most $8(N+1)\cdot m \leq 8(N+1){{N+3}\choose{3}}$ 
   multiplications and $8N+4$ additions in $K$. 
   
   Adding this to the cost of Gaussian elimination, we obtain the costs in the 
   statement of the proposition.
\end{proof}
\par
\noindent {\bf Step 4: Find the linear map to bring the image
into desired form.}
From Steps 1 to 3, we have calculated the degree-$N$ map
\begin{align*}
   \psi = (\psi_X, \psi_Y, \psi_Z, \psi_T) : 
   \K \rightarrow \widetilde{\K},
\end{align*}
with kernel $\langle R, S \rangle$. 
However, $\widetilde{\K}$ may not be in desired Kummer suface model. 
Therefore,
we must apply a linear map $\lambda: \widetilde{\K} \rightarrow \K'$ 
to obtain the $(N, N)$-isogeny 
\begin{align*}
   \varphi = \lambda \circ \psi: \K \rightarrow 
   \K', 
   \ \ (k_1,k_2,k_3,k_4) \mapsto (k'_1,k'_2,k'_3,k'_4)
\end{align*}
with kernel $\langle R,S \rangle$, where $\K'$ is in 
the desired form.
\par 
\noindent
{\bf Step 5: Find the image Kummer surface.} The 
final step is to find the parameters defining the 
image Kummer surface.
\par
The method for finding the scaling and image Kummer surface 
varies depending on what 
model we are working with. In this article we focus on 
the General and Fast Kummer surface models. 
\par
Over the next two sections, we go
into detail on how to find isogenies between Kummer surfaces 
in the General and Fast models, providing examples in the case 
$N = 5$.

\subsection{Construction of~$(N,N)$-isogenies on 
the General Kummer model}
\label{subsec:isogenies-general}

We consider the case where $\K = \kgen$ is in General 
Kummer form. 
We want to find an $(N,N)$-isogeny $\varphi: \kgen \rightarrow 
\hkgen$ with kernel $\langle R, S \rangle$ for $R, S \in \kgen[N]$.
Suppose we have followed Steps 1-3 to obtain a degree-$N$ map 
\begin{align*}
    \psi = (\psi_1, \psi_2, \psi_3, \psi_4) \colon \kgen 
    &\rightarrow \widetilde{\K}, \\
        (k_1, k_2, k_3, k_4) &\mapsto (\ell_1, \ell_2, \ell_3, \ell_4).
\end{align*}
We must now find the linear map $\lambda : \widetilde{\K} 
\rightarrow \hkgen$ to obtain the isogeny as 
$\varphi = \lambda \circ \psi$.
\subsubsection{Find the final linear map}
Examine the 
$k_1 k_4^{N-1}, k_2 k_4^{N-1}, k_3 k_4^{N-1}, k_4^N$ terms in 
each~$\psi_i$ and perform a linear map so that 
$\ell_1$ has only a $k_1 k_4^{N-1}$ term (but not the others), 
$\ell_2$ has only a $k_2 k_4^{N-1}$ term (but not the others), 
$\ell_3$ has only a $k_3 k_4^{N-1}$ term (but not the others), 
$\ell_4$ has only a $k_4^N$ term (but not the others). 
\par
We now wish to write the quartic which is satisfied
by $\ell_1,\ell_2,\ell_3,\ell_4$. To find the coefficients of 
this quartic, we  use the formal power series in~\Cref{k1k2k3k4ps}, 
and then check it is correct.
We expect the quartic to have the form
\begin{equation}\label{expectedquartic}
(\ell_2^2 - 4 \ell_1 \ell_3) \ell_4^2 
+ \mu_1(\ell_1,\ell_2,\ell_3) \ell_4 
+ \mu_0(\ell_1,\ell_2,\ell_3),
\end{equation}
where $\mu_1$ is cubic and $\mu_0$ is quartic. By this we mean:
initially all coefficients in $\mu_1(\ell_1,\ell_2,\ell_3)$
and $\mu_0(\ell_1,\ell_2,\ell_3)$ should be variables.
We then replace~$k_1,k_2,k_3,k_4$ with the power series
in~\Cref{k1k2k3k4ps}. Any true identity should make the power
series in $s_1$, $s_2$ equal to zero. This gives a set
linear equations in the coefficients which allow us to
solve quickly for $\mu_1(\ell_1,\ell_2,\ell_3)$
and $\mu_0(\ell_1,\ell_2,\ell_3)$.
\par
If the defining equation has terms of the form
$\ell_1 \ell_2^2 \ell_4$, $\ell_2^3 \ell_4$, $\ell_2^2 \ell_3 \ell_4$,
we apply a further linear map 
$
(\ell_1, \ell_2, \ell_3, \ell_4) \mapsto (\ell_1, \ell_2, \ell_3, \ell'_4),
$
where
$\ell'_4 = \ell_4 - u_1 \ell_1 - u_2 \ell_2 - u_3 \ell_3$,
so that the equation satisfied by $\ell_1,\ell_2,\ell_3,\ell'_4$
no longer has these terms.

\subsubsection{Find the image General Kummer surface}
For Step~5, we now use~\Cref{kummerequation} 
to read off ${f}'_0,\ldots,{f}'_6$ 
 as the coefficients of
$-4 k_1^3$, $-2 k_1^2 k_2$, $-4 k_1^2 k_3$, $-2 k_1 k_2 k_3$,
$-4 k_1 k_3^2 f_4$, $-2 k_2 k_3^2$, $-4 k_3^3$, respectively,
and then check that our quartic is indeed the General Kummer
equation for the target curve $y^2 = {f}'_6 x^2 + \ldots + {f}'_0$.
\par
\begin{example}
Let $\calC : y^2 = x^5 + 3 x^4 + 9 x^3 + 10 x^2 + 9 x + 3$
over the finite field~$\F_{11}$.
If we specialise~\Cref{kummerequation}, we see
that the General Kummer equation is
\begin{equation}\label{genkumeqnexample}
\begin{split}
\kgen \colon &(7 k_1 k_3 + k_2^2) k_4^2\\
&+ (10 k_1^3 + 4 k_1^2 k_2 + 4 k_1^2 k_3 
 + 4 k_1 k_2 k_3 + 10 k_1 k_3^2 + 9 k_2 k_3^2) k_4\\
&+ 5 k_1^4 + k_3^4 + 3 k_1^2 k_2 k_3 + 8 k_1 k_2^2 k_3 + 4 k_1 k_2 k_3^2
 + k_1^2 k_3^2\\ 
&+ 2 k_1^3 k2 + 3 k_1^3 k_3 + 8 k_1^2 k_2^2 + 10 k_1 k_2^3 + 4 k_1 k_3^3 = 0.
\end{split}
\end{equation}
There are two independent points of order~$5$, and their
images on the General Kummer are $R = (0,1,4,5)$
and $S = (0,1,0,0)$. Note also that $2R = (1,8,5,7)$
and $2S = (1,0,0,5)$. After applying Steps~1 to 3, we
obtain quintics, which are invariant under
addition by~$R$ and by~$S$. After adjusting them
linearly to have the correct terms for $k_1 k_4^4$, $k_2 k_4^4$,
$k_3 k_4^4$ and $k_4^5$, we obtain $\psi : (k_1, k_2, k_3, k_4)
\mapsto (\ell_1, \ell_2, \ell_3, \ell_4)$ where 

\begin{equation}\label{l1l2l3l4example}
    \scalebox{0.86}{{%
    $
    \begin{aligned}
        \ell_1 = \ &6 k_1^5+9 k_1^4 k_2+4 k_1^4 k_3+3 k_1^4 k_4
            +5 k_1^3 k_2^2+6 k_1^3 k_2 k_3+4 k_1^3 k_2 k_4\\
            &+5 k_1^3 k_3^2+7 k_1^3 k_3 k_4+10 k_1^3 k_4^2
            +6 k_1^2 k_2^3+k_1^2 k_2^2 k_3+8 k_1^2 k_2^2 k_4\\
            &+7 k_1^2 k_2 k_3^2+9 k_1^2 k_2 k_3 k_4+4 k_1^2 k_2 k_4^2
            +10 k_1^2 k_3^3+8 k_1^2 k_3^2 k_4+8 k_1^2 k_3 k_4^2\\
            &+7 k_1^2 k_4^3+7 k_1 k_2^2 k_3^2+2 k_1 k_2 k_3^3
            +4 k_1 k_2 k_3 k_4^2+9 k_1 k_2 k_4^3+2 k_1 k_3^4\\
            &+10 k_1 k_3^3 k_4+5 k_1 k_3^2 k_4^2+9 k_1 k_3 k_4^3
            +k_1 k_4^4+k_2^3 k_3^2+3 k_2^2 k_3^3+k_2^2 k_3^2 k_4\\
            &+3 k_2 k_3^3 k_4+8 k_2 k_3^2 k_4^2+7 k_3^4 k_4
            +7 k_3^3 k_4^2+6 k_3^2 k_4^3,\\
        \ell_2 = \ &3 k_1^5+10 k_1^4 k_2+9 k_1^4 k_3+7 k_1^4 k_4
            +9 k_1^3 k_2^2+7 k_1^3 k_2 k_3+2 k_1^3 k_2 k_4+10 k_1^3 k_3^2\\
            &+8 k_1^3 k_3 k_4+k_1^3 k_4^2+7 k_1^2 k_2^2 k_4
            +6 k_1^2 k_2 k_3^2+2 k_1^2 k_2 k_3 k_4+2 k_1^2 k_2 k_4^2\\
            &+8 k_1^2 k_3^3+9 k_1^2 k_3 k_4^2+9 k_1^2 k_4^3
            +2 k_1 k_2^3 k_3+k_1 k_2^2 k_3^2+4 k_1 k_2^2 k_3 k_4\\
            &+2 k_1 k_2 k_3^3+2 k_1 k_2 k_3^2 k_4
            +k_1 k_2 k_3 k_4^2+5 k_1 k_2 k_4^3+10 k_1 k_3^4+8 k_1 k_3^3 k_4\\
            &+9 k_1 k_3^2 k_4^2+3 k_1 k_3 k_4^3+2 k_2^3 k_3 k_4
            +10 k_2^2 k_3^3+10 k_2^2 k_3^2 k_4+2 k_2 k_3^4+4 k_2 k_3^3 k_4\\
            &+6 k_2 k_3^2 k_4^2+3 k_2 k_3 k_4^3+k_2 k_4^4+3 k_3^5
            +9 k_3^4 k_4+8 k_3^3 k_4^2,\\
        \ell_3 = \ &9 k_1^5+3 k_1^4 k_2+9 k_1^4 k_3+7 k_1^3 k_2 k_3
            +9 k_1^3 k_2 k_4+8 k_1^3 k_3^2+7 k_1^3 k_3 k_4+3 k_1^3 k_4^2+k_1^2 k_2^3\\
            &+k_1^2 k_2^2 k_3+5 k_1^2 k_2 k_3^2+8 k_1^2 k_2 k_3 k_4
            +7 k_1^2 k_2 k_4^2+5 k_1^2 k_3^3+2 k_1^2 k_3^2 k_4+7 k_1^2 k_3 k_4^2\\
            &+k_1^2 k_4^3+k_1 k_2^4+10 k_1 k_2^3 k_3+k_1 k_2^3 k_4
            +7 k_1 k_2^2 k_3^2+k_1 k_2^2 k_3 k_4+6 k_1 k_2 k_3^2 k_4\\
            &+7 k_1 k_2 k_3 k_4^2+2 k_1 k_2 k_4^3+6 k_1 k_3^4
            +4 k_1 k_3^3 k_4+9 k_1 k_3^2 k_4^2+9 k_1 k_3 k_4^3\\
            &+k_2^2 k_3^2 k_4+k_2 k_3^2 k_4^2+9 k_2 k_3 k_4^3
            +9 k_3^5+10 k_3^4 k_4+8 k_3^3 k_4^2+6 k_3^2 k_4^3+k_3 k_4^4,\\
        \ell_4 = \ &10 k_1^4 k_2+k_1^4 k_3+5 k_1^4 k_4+k_1^3 k_2^2+4 k_1^3 k_2 k_3
            +9 k_1^3 k_2 k_4+3 k_1^3 k_3^2+3 k_1^3 k_3 k_4+9 k_1^3 k_4^2\\
            &+10 k_1^2 k_2^3+8 k_1^2 k_2^2 k_3+9 k_1^2 k_2^2 k_4
            +5 k_1^2 k_2 k_3^2+10 k_1^2 k_2 k_3 k_4+6 k_1^2 k_2 k_4^2+10 k_1^2 k_3^3\\
            &+6 k_1^2 k_3^2 k_4+8 k_1^2 k_4^3+4 k_1 k_2^4
            +5 k_1 k_2^3 k_3+5 k_1 k_2^3 k_4+3 k_1 k_2 k_3^3+k_1 k_2 k_3^2 k_4\\
            &+5 k_1 k_2 k_3 k_4^2+10 k_1 k_2 k_4^3+8 k_1 k_3^4
            +8 k_1 k_3^3 k_4+6 k_1 k_3^2 k_4^2+4 k_1 k_3 k_4^3+9 k_1 k_4^4\\
            &+k_2^5+10 k_2^4 k_3+k_2^4 k_4+7 k_2^3 k_3^2
            +10 k_2^3 k_3 k_4+2 k_2^2 k_3^3+8 k_2^2 k_3^2 k_4+9 k_2 k_3^4\\
            &+5 k_2 k_3^2 k_4^2+10 k_2 k_3 k_4^3+9 k_2 k_4^4+2 k_3^5
            +10 k_3^4 k_4+6 k_3^3 k_4^2+2 k_3^2 k_4^3+9 k_3 k_4^4+k_4^5.
    \end{aligned}$}}
\end{equation}
We now find the quartic which is satisfied
by $\ell_1,\ell_2,\ell_3,\ell_4$.
We set up the form in~\Cref{expectedquartic} where,
for the moment, the coefficients of~$\mu_1$ and $\mu_0$
are variables. We then substitute~\Cref{k1k2k3k4ps}
into~\Cref{l1l2l3l4example}, and then substitute
these into~\Cref{expectedquartic}, truncating the power series
at the degree~$10$ terms in~$s_1,s_2$. The fact that the
power series is identically zero gives us a set of linear
equations which we solve in the
coefficients of~$\mu_1$ and $\mu_0$. This gives the following 
quartic in $\ell_1,\ell_2,\ell_3,\ell_4$.
\begin{equation}\label{quarticl1l2l3l4}
\begin{split}
&\ell_4^2 (\ell_2^2+7 \ell_1 \ell_3)\\
    &+ (3 \ell_1^3+5 \ell_1^2 \ell_2+7 \ell_1 \ell_2^2
      +9 \ell_1 \ell_2 \ell_3+7 \ell_1 \ell_3^2
    +7 \ell_2^3+7 \ell_2^2 \ell_3+9 \ell_2 \ell_3^2+2 \ell_3^3) \ell_4\\
    &+4 \ell_1^4+9 \ell_1^2 \ell_2^2+7 \ell_1^2 \ell_2 \ell_3
    +7 \ell_1^2 \ell_3^2 +8 \ell_1 \ell_2^2 \ell_3\\
    &+4 \ell_1 \ell_2 \ell_3^2+2 \ell_1 \ell_3^3+8 \ell_2^4+3 \ell_2^3 \ell_3
    +5 \ell_2^2 \ell_3^2+7 \ell_2 \ell_3^3+7 \ell_3^4.
\end{split}
\end{equation}
If we wish, we can verify that $\ell_1,\ell_2,\ell_3,\ell_4$
indeed satisfy this equation, by substituting~\Cref{l1l2l3l4example}
into~\Cref{quarticl1l2l3l4} and checking that the result
is indeed divisible by~\Cref{genkumeqnexample}.
\par
This is now very close to the style of a General Kummer equation, but note that we have
terms $\ell_1 \ell_2^2 \ell_4$, $\ell_2^3 \ell_4$, $\ell_2^2 \ell_3 \ell_4$
which do not appear in~\Cref{kummerequation}. 
We perform
the final linear map $(\ell_1, \ell_2, \ell_3, \ell_4)
\mapsto (\ell_1, \ell_2, \ell_3, \ell'_4)$, finding 
$\ell'_4$ as follows. 
Substitute $\ell_4 = 
{\ell}'_4 + u_1 \ell_1 + u_2 \ell_2 + u_3 \ell_3$
and solve the linear equations in~$u_1,u_2,u_3$ which make these
terms disappear. This gives $u_1 = u_2 = u_3 = 2$, and
so we define ${\ell}'_4 = \ell_4 - 2 \ell_1 - 2 \ell_2 - 2 \ell_3$.
The following quartic is satisfied by $\ell_1,\ell_2,\ell_3,{\ell}'_4$.
\begin{equation}\label{targetgenkummer}
\begin{split}
 &(7 \ell_1 \ell_3+\ell_2^2) (\ell'_4)^2\\
 &+(3 \ell_1^3 
  + 5 \ell_1^2 \ell_2 
  + 6 \ell_1^2 \ell_3
  +4 \ell_1 \ell_2 \ell_3
  +2 \ell_1 \ell_3^2
  +9 \ell_2 \ell_3^2
  +2 \ell_3^3) 
  \ell'_4\\
 &+4 \ell_2^4+6 \ell_2^3 \ell_3+4 \ell_1^2 \ell_2^2
 +3 \ell_1 \ell_2^3+5 \ell_1^3 \ell_2\\
 &+8 \ell_2^2 \ell_3^2+\ell_1^3 \ell_3+3 \ell_1^2 \ell_2 \ell_3
 +2 \ell_1 \ell_2^2 \ell_3.
\end{split}
\end{equation}
We now read off the coefficients of
$-4 k_1^3$, $-2 k_1^2 k_2$, $-4 k_1^2 k_3$, $-2 k_1 k_2 k_3$,
$-4 k_1 k_3^2 f_4$, $-2 k_2 k_3^2$, $-4 k_3^3$,
to see that ${f}'_0 = 2$, ${f}'_1 = 3$, ${f}'_2 = 4$,
${f}'_3 = 9$, ${f}'_4 = 5$, ${f}'_5 = 1$
and ${f}'_6 = 5$. We then confirm 
that~\Cref{targetgenkummer} is indeed the General Kummer
equation for the curve $y^2 = 5 x^6 + x^5 + 5 x^4 + 9 x^3 + 4 x^2 + 3 x + 2$,
and the $(5,5)$-isogeny is given by
$(k_1,k_2,k_3,k_4) \mapsto (\ell_1,\ell_2,\ell_3,\ell'_4)$.
We have determined the $(5,5)$-isogeny and the equation 
of the target General Kummer, as required.
\end{example}

\subsection{Construction of~$(N,N)$-isogenies on 
the Fast Kummer model}
\label{subsec:isogenies-fast}
Suppose we are given Fast Kummer $\kgaudry_{a,b,c,d}$, 
as in~\Cref{Gaudryeqn}. To follow notation in previous 
literature (e.g.,~\cite{GaudryFastGenus2}), 
we will let $(X, Y, Z, T)$ be the coordinates of $\kgaudry_{a,b,c,d}$ rather than $k_1, \dots, k_4$ as before.
We wish to compute the $(N,N)$-isogeny 
$\varphi: \kgaudry_{a,b,c,d} \rightarrow 
\kgaudry_{a',b',c',d'}$,
$(X,Y,Z,T) \mapsto (X', Y', Z', T')$,
with kernel generated by $N$-torsion points $R, S \in \kgaudry$.
\par
We first note that there will be many choices of such
an isogeny, since there is a rich set of linear maps
between Fast Kummer surfaces, and any~$(N,N)$-isogeny can
be composed with any of these to get a variant~$(N,N)$-isogeny.
\par
Addition by any of the~16 points of order~$2$
give the linear maps $\sigma_i$ defined 
by~\Cref{order2actionGaudry}
from $\kgaudry_{a,b,c,d}$ to itself.
Furthermore, there is the \emph{Hadamard} map
\begin{equation}\label{linearGaudry}
\begin{split}
\mathcal{H} : \kgaudry_{a,b,c,d}
&\rightarrow 
\kgaudry_{a',b',c'd'}\\
(X,Y,Z,T) &\mapsto (X+Y+Z+T, X+Y-Z-T, \\
& \hspace{20ex} X-Y+Z-T, X-Y-Z+T),
\end{split}
\end{equation}
where $(a',b',c',d') = \mathcal{H}(a,b,c,d)$.
All of the above are defined over the ground field~$K$
and, given any $(N,N)$-isogeny between Fast Kummers,
any of the above (on the target Kummer) can be composed
with it to give a variant $(N,N)$-isogeny. 
\par
In the special case when there exists $i\in K$, where $i^2 = -1$,
(for example, when $K = \F_p$ for $p\equiv 1 \bmod 4$),
there is also the map $(X,Y,Z,T) \mapsto (X,Y,iZ,iT)$
from $\kgaudry_{a,b,c,d}$ to $\kgaudry_{a,b,ic,id}$.
\par
Since our~$(N,N)$-isogeny has odd degree, it must
be an isomorphism on the entire
$2$-torsion subgroup of $\kgaudry_{a,b,c,d}$. 
After composing with a combination of the linear maps above,
we can force the diagonalised elements of order~$2$ to
map to the diagonalised elements of order~$2$ on the target
Fast Kummer and match them up so that they have the same effect.
In other words, we can choose our $(N,N)$-isogeny so
that $(X,Y,Z,T) \mapsto (X,Y,-Z,-T)$, $(X,Y,Z,T) \mapsto (X,-Y,Z,-T)$
and $(X,Y,Z,T) \mapsto (X,-Y,-Z,T)$ have the same effect on
the target coordinates $X',Y',Z',T'$. 
So, we should be able to choose our isogeny so that the monomials
can be partitioned in the same way as described in the previous section for
$\phi_X^{(N)},\phi_Y^{(N)},\phi_Z^{(N)},\phi_T^{(N)}$.
Similarly $(X,Y,Z,T) \mapsto (Y,X,T,Z)$, $(X,Y,Z,T) \mapsto (Z,T,X,Y)$
and $(X,Y,Z,T) \mapsto (T,Z,Y,X)$ should have the same effect on
the target coordinates $X',Y',Z',T'$ (if only projectively, then
applying any of these involutions twice forces the scalar
to be~$\pm 1$). We will use this symmetry 
throughout this section to 
find the scaling, as well as accelerate~\Cref{alg:findintersection} when using 
Fast Kummer surfaces.
\begin{remark}
    The biquadratic forms corresponding to the 
    Fast Kummer surface $\kgaudry$ are the simplest, when compared to 
    those for $\ksqr$ and $\kgen$. Therefore, when constructing our
    $(N,N)$-isogenies from these biqudratics, we expect the 
    isogenies between Fast Kummer surfaces to yield the most 
    efficient and compact maps. 
\end{remark}
We now describe a method for computing these isogenies
which finds the most natural version, namely
the version in which addition by the points of order~$2$
on the initial Fast Kummer surface have the same effect on the
coordinates of the target Fast Kummer surface. 
\par
Our input is: the parameters~$a,b,c,d$ of the
initial Fast Kummer, together with two independent
points~$R,S$ of order~$N$, where~$N$ is odd.
We will output the formul\ae{} defining the 
$(N,N)$-isogeny $\varphi: \kgaudry_{a,b,c,d} 
\rightarrow \kgaudry_{a',b',c',d'}$ that takes
$(X,Y,Z,T)$ to $(X',Y',Z',T')$ 
with kernel $\langle R, S \rangle$.
\par
\subsubsection{Modification to \textsf{FindIntersection}}
We first discuss how we can accelerate \textsf{FindIntersection}
(see~\Cref{alg:findintersection}) using the 
action of the $2$-torsion points on $\kgaudry$. Suppose 
we have followed Step 1 in~\Cref{section5}
to find the
space $\spaceR$ of degree~$N$ homogeneous forms which 
are invariant under
addition by~$R$ and the space $\spaceS$ of those 
invariant under addition by~$S$.
Note that we are now using the
simpler biquadratic forms in~\Cref{biquadformscorollary}. 
Let $\basisR$ and $\basisS$ be a 
basis for these spaces, respectively.
\par 
\noindent
Rather than proceeding straight to Step 2, 
we partition the basis $\basisR$ into four parts
$\basisRi{1}$, $\basisRi{2}$, $\basisRi{3}$ and $\basisRi{4}$, as
follows:
\begin{align*}
   \basisRi{1} &:= \big\{ f_R \in \basisR \ | \ \sigma_1(f_R) = f_R, 
   \  \sigma_2(f_R) = f_R, \ \sigma_{3}(f_R) = f_R \big\} \\
   \basisRi{2} &:= \big\{ f_R \in \basisR \ | \ \sigma_1(f_R) = f_R, 
   \  \sigma_2(f_R) = -f_R, \ \sigma_{3}(f_R) = -f_R \big\} \\
   \basisRi{3} &:= \big\{ f_R \in \basisR \ | \ \sigma_1(f_R) = -f_R, 
   \  \sigma_2(f_R) = f_R, \ \sigma_{3}(f_R) = -f_R \big\} \\
   \basisRi{4} &:= \big\{ f_R \in \basisR \ | \ \sigma_1(f_R) = -f_R, 
   \  \sigma_2(f_R) = -f_R, \ \sigma_{3}(f_R) = f_R \big\},
\end{align*}
where $\sigma_i$ is the linear map corresponding to 
the action of two-torsion point $E_i$, as defined 
by~\Cref{order2actionGaudry}. 
Refering to the explicit basis given in~\Cref{conj:basis-space},
the parts $B^{(1)}$, $B^{(2)}$, $B^{(3)}$, $B^{(4)}$ of the basis $B$ 
are given by the forms corresponding to the indices in $\mathcal{I}_N^{(1)}$, 
$\mathcal{I}_N^{(2)}$, $\mathcal{I}_N^{(3)}$, $\mathcal{I}_N^{(4)}$, respectively,
where
\begin{align*}
   \mathcal{I}_N^{(1)} &:= \Big\{ \{ i_1, \dots, i_N \} \in \mathcal{I}_N \ | \ i_j \in \{1, 2\} \text{ and there are an even number of } 2\text{'s} \Big\}, \\
   \mathcal{I}_N^{(2)} &:= \Big\{ \{ i_1, \dots, i_N \} \in \mathcal{I}_N \ | \ i_j \in \{1, 2\} \text{ and there are an odd number of } 2\text{'s} \Big\}, \\
   \mathcal{I}_N^{(3)} &:= \Big\{ \{ i_1, \dots, i_N \} \in \mathcal{I}_N \ | \ i_j \in \{3, 4\} \text{ and there are an even number of } 4\text{'s} \Big\}, \\
   \mathcal{I}_N^{(4)} &:= \Big\{ \{ i_1, \dots, i_N \} \in \mathcal{I}_N \ | \ i_j \in \{3, 4\} \text{ and there are an odd number of } 4\text{'s} \Big\}.
\end{align*}
Using this characterisation, we can compute the partition as we run 
\textsf{FindBasis} (see~\Cref{alg:findbasis}) 
and immediately output $\basisRi{1}, \basisRi{2}, 
\basisRi{3}$, and $\basisRi{4}$ for $R$.
We proceed similarly with $S$ to obtain $\basisSi{1}$, 
$\basisSi{2}$, $\basisSi{3}$ and $\basisSi{4}$.
\par 
For each part $i = 1,\dots, 4$,
intersect the spaces $\spaceRi{i}$ and $\spaceSi{i}$ (generated by 
$\basisRi{i}$ and $\basisSi{i}$, respectively), to obtain 
$\spaceRSi{i}$. We expect each $\spaceRSi{i}$ to be of dimension~$1$, 
and thus generated by a homogenous form, which will give the $i$-th 
coordinate of degree-$N$ map
$\psi = (\psi_X, \psi_Y, \psi_Z, \psi_T)$.
We can find this intersection using \textsf{FindIntersectionPart},
which on input $\basisRi{i}, \basisSi{i}$, 
will output the generator of $\spaceRSi{i}$. 
It runs identically to \textsf{FindIntersection}
except the dimensions of the linear system decrease by 
a factor of $4$. As a result, \textsf{FindIntersectionPart} terminates in $\frac{1}{24} (N + 1)(N + 12)(N - 1)$ {\tt M} and $\frac{1}{24} (N + 1)(N + 6)(N - 1)$ {\tt a}. Running this for each part we 
terminate in at most $\frac{1}{6} (N + 1)(N + 12)(N - 1)$ {\tt M} and $\frac{1}{6} (N + 1)(N + 6)(N - 1)$ {\tt a} which improves on the cost quoted in~\Cref{prop:findintersection}, and is concretely 
much faster.

\subsubsection{General method for computing the final linear map}
At this stage, we have calculated the degree-$N$ map
\begin{align*}
   \psi = (\psi_X, \psi_Y, \psi_Z, \psi_T) : 
   \kgaudry_{a,b,c,d} \rightarrow \widetilde{\K},
\end{align*}
however $\widetilde{\K}$ may not be in desired Fast Kummer form. Therefore,
we must apply a scaling map $$\lambda: (X,Y,Z,T) \mapsto 
 (\lambda_X X, \lambda_Y Y, \lambda_Z Z, \lambda_T T),$$
to obtain the 
$(N, N)$-isogeny 
\begin{align*}
   \varphi = \lambda \circ \psi: \kgaudry_{a,b,c,d} \rightarrow 
   \kgaudry_{a',b',c',d'}, 
   \ \ (X, Y, Z, T) \mapsto (X', Y', Z', T')
\end{align*}
with kernel $\langle R,S \rangle$.

To find the scaling values, we exploit the fact that the 
actions
\begin{align*}
   \sigma_4 \colon &(X,Y,Z,T) \mapsto (Y,X,T,Z), \\
   \sigma_7 \colon &(X,Y,Z,T) \mapsto (Z,T,X,Y), \\
   \sigma_{12} \colon &(X,Y,Z,T) \mapsto (T,Z,Y,X),
\end{align*}
of two-torsion points $E_4 = (b,a,d,c)$, $E_8 = (c,d,a,b)$, 
and $E_{12} = (d,c,b,a)$ on $\kgaudry_{a,b,c,d}$,
have the same effect on the target coordinates.
From this we obtain the following 
linear equations, which we can solve to obtain 
$(\lambda_X,\lambda_Y,\lambda_Z,\lambda_T)$:
\begin{align}\label{eq:2torsaction-scaling}
   \varphi\big(\sigma^{\K}_{i}(X,Y,Z,T)\big) &= 
   \sigma^{\K'}_{i}\big(\varphi(X,Y,Z,T)\big), \text{ for } i = 4,8,12,
\end{align}
where $\sigma^{\K}$ denotes the action of two-torsion points  
on $\kgaudry_{a,b,c,d}$, and $\sigma^{\K'}$ on 
$\kgaudry_{a',b',c',d'}$. 
Importantly, however, this equality only holds modulo the equation 
defining the domain Kummer surface.

For $N = 3$, the scaling map is given by Corte-Real Santos, Costello and Smith 
in~\cite[\S 4]{CorteRealSantosCostelloSmithIsogenies}. In this case, finding the 
scaling map is straightforward as the isogeny formul\ae{} is unaffected by working 
modulo the equation defining $\kgaudry_{a,b,c,d}$ of degree $4$. 
However, for odd $N \geq 5$, 
this no longer holds, and we must take care when finding this scaling. 

In this setting, we discuss three distinct methods to compute the scaling map:
\begin{enumerate}
   \item a method for $N = 5$ which requires $62$ $K$-multiplications.
   \item a method for $N \geq 7$, which requires running Gaussian elimination on a 
   system of $\ell$ equations in $\ell+1$ unknowns, where $\ell \leq 
   (N-1)(N-2)(N-3)/24$. 
   \item a method for $N \geq 7$, which requires the computation of $2$ square roots 
   in $K$, $1$ inverse in $K$ and a few $K$-multiplications.
\end{enumerate}

\subsubsection{Final scaling for $N = 5$}

We start by describing the first method for $N=5$, summarised by~\Cref{alg:scaling5}. 
We consider the first coordinate in~\Cref{eq:2torsaction-scaling}
taking $i = 4$, to obtain the equality
$$\lambda_X \psi_X(X,Y,Z,T) = \lambda_Y \psi_Y(Y,X,T,Z) + cX\kgaudry(X,Y,Z,T),$$
for some constant $c\in K$, where by abuse of notation $\kgaudry(X,Y,Z,T)$
represents the equation of the Kummer surface.
We say that coefficients are \emph{transparently equal} if the corresponding 
monomials are unaffected by the Kummer equations. In this way, we see
that the coefficients of the $YZT^3$ term are 
transparently equal. More 
precisely, the coefficient of $YZT^3$ in $\lambda_X \psi_X(X,Y,Z,T)$ should
be equal to the coefficient of $YZT^3$ in $\lambda_Y \psi_Y(Y,X,T,Z)$, 
namely the coefficient of $XZ^3T$ in $\lambda_Y \psi_Y(X,Y,Z,T)$. 
As we are working projectively, we can set $\lambda_X = 1$, and 
thus deduce the value of $\lambda_Y$. 
We proceed similarly using~\Cref{eq:2torsaction-scaling} with 
$i = 8$ and $12$, to obtain the scaling values 
$\lambda_Z$ and $\lambda_T$, respectively.

\begin{algorithm}
	\caption{$\textsf{Scaling}_5$: find the scaling map when $N=5$}
   \label{alg:scaling5}
	\begin{algorithmic}[1]
        \REQUIRE{A quintic map $\psi: \kgaudry \rightarrow \widetilde{\K}$}
        \ENSURE{The isogeny $\varphi: \kgaudry_{a,b,c,d}\rightarrow \kgaudry_{a',b',c',d'}$
        with kernel generated by $R,S$}
        \STATE $\psi = (\psi_X, \psi_Y, \psi_Z, \psi_T)$
         \STATE Let $c_{Y}$ be the coefficient of $YZT^3$ in $\psi_X$
         \STATE Let $d_{Y}$ be the coefficient of $XZ^3T$ in $\psi_Y$
         \STATE Let $c_{Z}$ be the coefficient of $YZT^3$ in $\psi_X$
         \STATE Let $d_{Z}$ be the coefficient of $XY^3T$ in $\psi_Z$
         \STATE Let $c_{T}$ be the coefficient of $XYT^3$ in $\psi_Z$
         \STATE Let $d_{T}$ be the coefficient of $XYZ^3$ in $\psi_T$
         \STATE $\alpha = d_Z\cdot d_T$\label{line:scaling5:firstmult}
         \STATE $\beta = d_Y\cdot c_Z$
        \STATE $(\lambda_X, \lambda_Y, \lambda_Z, \lambda_T) = 
        (d_Y \alpha, \ c_Y  \alpha, \ 
        d_T  \beta, \ c_T  \beta)$\label{line:scaling5:lastmult}
        \STATE $\varphi = (\lambda_X \psi_X, \lambda_Y \psi_Y, \lambda_Z \psi_Z, 
        \lambda_T \psi_T)$\label{line:scaling5:varphi}
        \RETURN $\varphi$
    \end{algorithmic} 
\end{algorithm}

\begin{proposition}\label{prop:scaling5}
   \Cref{alg:scaling5} terminates with $62$ {\tt M}.
\end{proposition}
\begin{proof}
    Lines~\ref{line:scaling5:firstmult} to~\ref{line:scaling5:lastmult}
    in~\Cref{alg:scaling5} 
    cost $6$ {\tt M}. Then, constructing the output isogeny $\varphi$ in
    Line~\ref{line:scaling5:varphi} requires $56$ {\tt M}.
\end{proof}

\begin{example} We now illustrate the method for 
$N = 5$ with the following example. 
Our input is the Fast Kummer surface defined by parameters
$(a,b,c,d) = (883, 375, 1692, 1586)$ over the finite field $\F_{1697}$,
together with the points $R = (1593, 713, 1161, 1)$
and $S = (615, 1249, 125, 1)$ of order~$5$. We wish to compute the 
$(5,5)$-isogeny $\varphi$ with kernel $\langle R, S \rangle$.

After applying Steps~$1$ to $3$, we obtain the quintic map 
$\psi = (\psi_X, \psi_Y, \psi_Z, \psi_T)$, where:
\begin{equation}\label{unscaledquinticmap}
\begin{split}
\phi_X =& \, 1668 X^5+708 X^3 Y^2+1282 X^3 Z^2+1487 X^3 T^2\\
&+823 X^2 Y Z T+646 X Y^4+760 X Y^2 Z^2+502 X Y^2 T^2+632 X Z^4\\
&+1352 X Z^2 T^2+247 X T^4+651 Y^3 Z T+1154 Y Z^3 T+331 Y Z T^3,\\
\phi_Y =& \, 1026 X^4 Y+512 X^3 Z T+556 X^2 Y^3+278 X^2 Y Z^2\\
&+7 X^2 Y T^2+509 X Y^2 Z T+289 X Z^3 T+136 X Z T^3+370 Y^5\\
&+975 Y^3 Z^2+1564 Y^3 T^2+329 Y Z^4+1026 Y Z^2 T^2+942 Y T^4,\\
\phi_Z =& \, 259 X^4 Z+19 X^3 Y T+1373 X^2 Y^2 Z+396 X^2 Z^3\\
&+686 X^2 Z T^2+1101 X Y^3 T+1610 X Y Z^2 T+371 X Y T^3+660 Y^4 Z\\
&+1520 Y^2 Z^3+1539 Y^2 Z  T^2+229 Z^5+933 Z^3 T^2+121 Z T^4,\\
\phi_T =& \, 397 X^4 T+371 X^3 Y Z+610 X^2 Y^2 T+1326 X^2 Z^2 T\\
&+1464 X^2 T^3+1073 X Y^3 Z+945 X Y Z^3+686 X Y Z T^2+80 Y^4 T\\
&+1613 Y^2 Z^2 T+816 Y^2 T^3+593 Z^4 T+770 Z^2 T^3+708 T^5.
\end{split}
\end{equation}
At the stage, we only require a further scaling of each of these
to obtain ~$X',Y',Z',T'$ such that the $(5,5)$-isogeny 
is defined by 
$$\varphi: \kgaudry_{a,b,c,d} \rightarrow 
\kgaudry_{a',b',c',d'}, \ \ 
(X,Y,Z,T) \mapsto (X',Y',Z',T').$$ 
Let us now apply Step~$4$.
We first let $\lambda_X,\lambda_Y,\lambda_Z,\lambda_T$ denote
the scaling factors, so that our desired~$X',Y',Z',T'$ will be
\begin{equation}\label{relationunscaledtoscaled}
X' = \lambda_X \psi_X,\ Y' = \lambda_Y \psi_Y,\ 
Z' = \lambda_Z \psi_Z,\ T' = \lambda_T \psi_Z.
\end{equation}
We now use that we are making the choice of map such
that $(X,Y,Z,T) \mapsto (Y,X,T,Z)$ has the same effect
on $(X',Y',Z',T')$. Applying $(X,Y,Z,T) \mapsto (Y,X,T,Z)$
to $\lambda_Y \psi_Y$ and equating one of its coefficients
to the corresponding coefficient in $\lambda_X \psi_X$, we
obtain linear equation in $\lambda_X,\lambda_Y$.
However, some care is required, since all equalities are
modulo the initial Fast Kummer equation. For quintics within
the first partition, this can only mean: modulo a constant
time $X \kgaudry_{a,b,c,d}(X,Y,Z,T)$, where here 
$\kgaudry_{a,b,c,d}(X,Y,Z,T)$ is the equation of the Kummer 
equation. So, for example, the~$Y Z T^3$
term is unaffected by this. Hence the coefficients of the~$Y Z T^3$
terms should be transparently equal. The coefficient of~$Y Z T^3$
in $\lambda_X \psi_X$ is~$331 \lambda_X$. The coefficient
of~~$Y Z T^3$ in the image of $\lambda_Y \psi_Y$
under $(X,Y,Z,T) \mapsto (Y,X,T,Z)$ (which is the same
as the coefficient of~$X Z^3 T$ in $\lambda_Y \psi_Y$) is $289\lambda_Y$.
This gives the equation $289 \lambda_Y = 331 \lambda_X$
and so $\lambda_Y = 283 \lambda_X$ in our field $\F_{1697}$.
We can now check that indeed if we take the entirety of~$283$ times
the image of $283 \psi_Y$ under $(X,Y,Z,T) \mapsto (Y,X,T,Z)$
and then subtract~$\psi_X$, the result is divisible 
by~$\kgaudry_{a,b,c,d}$, where $a = 883, b = 375, c = 1692, d = 1586$
are our inputted initial parameter values.
\par
Similarly, the coefficient of $X Y T^3$ in $\lambda_Z \psi_Z$
is $371 \lambda_Z$. The coefficient of $X Y T^3$ in
the image of $\lambda_T \psi_T$
under $(X,Y,Z,T) \mapsto (Y,X,T,Z)$ (which is the same
as the coefficient of~$X Y Z^3$ in $\lambda_T \psi_T$)
is $945\lambda_T$. Within the third partition set, these cannot
be affected by multiples of $\kgaudry_{a,b,c,d}$,
and so $371 \lambda_Z = 945\lambda_T$, giving that
$\lambda_T = 1270 \lambda_Z$ in our field $\F_{1697}$.
\par
We can also equate the coefficient of $Y Z T^3$
in $\lambda_X \psi_X$ with the same coefficient in the
image of $\lambda_Z \psi_Z$ under $(X,Y,Z,T) \mapsto (Z,T,X,Y)$
to see that $\lambda_Z = 418 \lambda_X$.
We have now solved for the projective array of scaling factors:
\begin{equation}\label{scalefactorsolution}
(\lambda_X, \lambda_Y, \lambda_Z, \lambda_T) = (1, 283, 418, 1396),
\end{equation}
and we have found the final versions of our maps:
\begin{equation}
X' = \psi_X,\ Y' = 283\psi_Y,\ Z' = 418\psi_Z,\ T' = 1396\psi_T.
\end{equation}
Thus, we have found the $(5,5)$-isogeny 
$$\varphi:  (X,Y,Z,T) \mapsto (X',Y',Z',T'),$$
 completing Step~$4$. For the final step, we 
compute $\varphi(a,b,c,d)$ by substituting 
$(X,Y,Z,T) = (883, 375, 1692, 1586)$
into $(X',Y',Z',T')$ to give that
$(a',b',c',d') = (381, 960, 69, 1199)$. 
\end{example}

\subsubsection{Final scaling for $N \geq 7$ with Gaussian elimination}\label{subsubsec:scalingGE-fast}
The next method to find the scaling for $N \geq 7$ follows the same procedure, 
however now there are no monomials whose coefficients will be transparently equal. 
Rather, they will be equal modulo the equation defining the domain Kummer surface. 
Instead, we must now perform Gaussian elimination to solve a system 
of equations for the scaling values $\lambda_X, \lambda_Y, \lambda_Z$ and $\lambda_T$. 
More precisely, we again consider~\Cref{eq:2torsaction-scaling} and look 
at the first coordinate to obtain  
\begin{align}
   \lambda_X \psi_X(X,Y,Z,T) - \lambda_Y \psi_Y(Y,X,T,Z) + G_Y\kgaudry(X,Y,Z,T) &= 0, \\
   \lambda_X \psi_X(X,Y,Z,T) - \lambda_Z \psi_Z(Z,T,X,Y) + G_Z\kgaudry(X,Y,Z,T) &= 0,\\
   \lambda_X \psi_X(X,Y,Z,T) - \lambda_T \psi_T(T,Z,Y,X) + G_T\kgaudry(X,Y,Z,T) &= 0,
\end{align}
where $G_Y, G_Z, G_T$ are $(N-4)$-degree forms such that 
$G_{\bullet}\kgaudry$ contains monomials in the first partition (i.e., are 
unaffected by the action of $\sigma_1, \sigma_2$ and $\sigma_3$). 
As this equality holds for any point $(X,Y,Z,T)$ on $\kgaudry_{a,b,c,d}$,
we must have that the coefficients of each monomial are identically zero. 
This gives us a system of equations that we solve to obtain the scaling 
factors $\lambda_X, \lambda_Y, \lambda_Z$, and $\lambda_T$. We remark again
that we can set $\lambda_X = 1$. 
\par
We summarise this method in~\Cref{alg:scalingechelon}. 
In Line~\ref{line:scalingechelon:mons}, on input an integer $n \geq 3$,
$\textsf{MonomialsInFirstPartition}(n)$, outputs the monomials of the $n$-degree 
homogenous forms in the `first partition' (i.e., those unchanged by the 
action of $\sigma_1, \sigma_2$ and $\sigma_3$). Note that these
monomials can be precomputed
for each $N$; indeed, they do not depend on the kernel 
generators $R$ and $S$.  
In Line~\ref{line:scalingechelon:echelonform}, on input a matrix $M$, the 
algorithm $\textsf{EchelonForm}$, outputs
the matrix $M$ in echelon form.

\begin{algorithm}
	\caption{$\textsf{Scaling}_{\text{GE}}$: find the scaling map when $N\geq 7$
   using Gaussian elimination}
   \label{alg:scalingechelon}
	\begin{algorithmic}[1]
        \REQUIRE{A map of degree~$N$, $\psi : \kgaudry_{a,b,c,d} \rightarrow \widetilde{\K}$,
        for $N \geq 7$.}
        \ENSURE{The isogeny $\varphi: \kgaudry_{a,b,c,d}\rightarrow \kgaudry_{a',b',c',d'}$
        with kernel generated by $R,S$.}
        \STATE $\psi = (\psi_X, \psi_Y, \psi_Z, \psi_T)$
        \STATE ${\tt mons} = \textsf{MonomialsInFirstPartition}(N-4)$\label{line:scalingechelon:mons}
        \FOR{$k = 2$ to $4$}
            \STATE $G = g_1m_1 + \cdots + g_{\ell}m_{\ell}$ for monomials $m_j \in {\tt mons}$
            \STATE Let $\K(X,Y,Z,T)$ be the equation defining $\kgaudry_{a,b,c,d}$
            \STATE $F = \lambda_k\psi_k(Y,X,T,Z) - \lambda_1\psi_X(X,Y,Z,T) + 
            G\K$
            \FOR{$j = 1$ to $\ell$}
               \STATE Let $c_{j} \in K[\lambda_k, g_1, \dots, g_{\ell}, \lambda_1]$ 
               be the coefficient of $m_j$ in $F$
            \ENDFOR
            \STATE $\mathbf{v} = (\lambda_k, g_1, \dots, g_{\ell}, \lambda_1)$
            \FOR{$ i = 1$ to $\ell + 1$}
               \STATE Let $c_{i,j}$ be the coefficient of $v_i \in \mathbf{v}$ in $c_{j}$
            \ENDFOR
            \STATE $M = (c_{i,j})_{i = 1, \dots, \ell + 1, j = 1, \dots, \ell}$ 
            \STATE $M = \textsf{EchelonForm}(M)$\label{line:scalingechelon:echelonform}
            \STATE $m = \textsf{NumberOfColumns}(M)$
            \STATE $\lambda_k = M_{1, m}$
        \ENDFOR
        \STATE $\varphi = (\psi_X, \lambda_2\psi_Y, 
        \lambda_3\psi_Z, \lambda_4\psi_T)$\label{line:scalingGE:varphi}
        \RETURN $\varphi$
    \end{algorithmic} 
\end{algorithm}

\begin{proposition}\label{prop:scalingGE}
   \Cref{alg:scalingechelon} terminates in $\cost{GE}$,
   where 
   $$\cost{GE} \leq \frac{1}{6}\ell(\ell + 1)(2\ell + 13) \, {\tt M}  + 
   \frac{1}{6}\ell(\ell + 1)(2\ell + 7) \, {\tt a},$$ 
   where $\ell \leq (N-1)(N-2)(N-3)/24.$
\end{proposition}
\begin{proof}
   To obtain the scaling factors $\lambda_X, \lambda_Y, \lambda_Z$, and $\lambda_T$, 
   we solve a system of $\ell+1$ equations 

   in $\ell+1$ unknowns, where 
   $$\ell \leq \frac{1}{4}{{N-1}\choose{3}} = \frac{(N-1)(N-2)(N-3)}{24}$$ 
   is the number of monomials of degree $N-4$ in the homogeneous forms 
   that are left unchanged by the action of $\sigma_1, \sigma_2$ and $\sigma_3$.
   Therefore, as in the proof of~\Cref{prop:findintersection}, the cost of solving this system using 
   Gaussian elimination is  
   $\frac{1}{6}\ell(\ell + 1)(2\ell + 13)$ multiplications and $\frac{1}{6}\ell(\ell + 1)(2\ell + 7)$ additions in $K$.

   In Line~\ref{line:scalingGE:varphi}, we calculate 
   $(\psi_X, \lambda_Y\psi_Y, \lambda_Z\psi_Z, 
   \lambda_T\psi_T)$ with at most $3\ell$ $K$-multiplications, 
   where $\ell$ is the maximum
   of the number of monomials in $\psi_Y, \psi_Z$ and
   $\psi_T$. Namely, $\ell \leq (N+1)(N+2)(N+3)/24$.

\end{proof}

\subsubsection{Final scaling for $N \geq 7$ with square roots}\label{subsubsec:scalingsqrt-fast}
We are now ready to describe the final method for $N \geq 7$. 
Here, we consider~\Cref{relationunscaledtoscaled} when evaluated at 
$(a,b,c,d)$. Namely,
\begin{equation}\label{relationunscaledtoscaled:abcd:sigma0}
   \begin{split} 
    (a',b',c',d') = \big(\lambda_X \psi_X(a,b,c,d)&, \ \lambda_Y \psi_Y(a,b,c,d), \\ 
    & \lambda_Z \psi_Z(a,b,c,d), \ \lambda_T \psi_T(a,b,c,d)\big).
   \end{split}
\end{equation} 
Applying the action $\sigma_1$, we also have that 
\begin{equation}\label{relationunscaledtoscaled:abcd:sigma1}
    \begin{split}
    (b',a', d', c') = \big(\lambda_X \psi_X(b,a,d,c)&, \lambda_Y \psi_Y(b,a,d,c), \\
    & \lambda_Z \psi_Z(b,a,d,c), \lambda_T \psi_T(b,a,d,c)\big).
    \end{split}
\end{equation} 
Looking at the first and second coordinates in the equations above, 
we have $$\frac{\lambda_X \psi_X(a,b,c,d)}
{\lambda_Y \psi_Y(a,b,c,d)} = \frac{a'}{b'} \text{ and } 
\frac{\lambda_X \psi_X(b,a,d,c)}
{\lambda_Y \psi_Y(b,a,d,c)} = \frac{b'}{a'}.$$
So, we see that
\begin{equation}\label{lambdaXY_squared}
   \Bigg( \frac{\lambda_Y}{\lambda_X} \Bigg)^2 = 
   \frac{\psi_X(a,b,c,d)\psi_X(b,a,d,c)}
   {\psi_Y(a,b,c,d)\psi_Y(b,a,d,c)},
\end{equation} 
and 
\begin{equation}\label{ab_squared}
   \Bigg( \frac{b'}{a'} \Bigg)^2 = 
\frac{\psi_X(b,a,d,c)\psi_Y(a,b,c,d)}
{\psi_X(a,b,c,d)\psi_Y(b,a,d,c)}.
\end{equation} 
Looking at the third and fourth coordinates of 
equations~\Cref{relationunscaledtoscaled:abcd:sigma0}
and~\Cref{relationunscaledtoscaled:abcd:sigma1}, we obtain 
$(\lambda_T/\lambda_Z)^2$ and $(d'/c')^2$.
Similarly, we can apply the action $\sigma_3$ to obtain 
$(\lambda_Z/\lambda_Y)^2 $, 
$\lambda_X\lambda_Z/(\lambda_Y\lambda_T)$, 
$(c'/b')^2$, and 
$a'c'/(b'd')$.

From $(\lambda_Y/\lambda_X)^2$, $(\lambda_T/\lambda_Z)^2$, 
and $\lambda_X\lambda_Z/(\lambda_Y\lambda_T)$, 
we can obtain the scalings 
$\lambda_X$, $\lambda_Y$, $\lambda_Z$, and $\lambda_T$ as 
described in~\Cref{alg:scalingsqrt}.

\begin{algorithm}
	\caption{$\textsf{Scaling}_{\text{sqrt}}$: find the scaling map when $N\geq 7$
   using square roots in $K$}
   \label{alg:scalingsqrt}
	\begin{algorithmic}[1]
        \REQUIRE{A map of degree~$N$, $\psi: \kgaudry_{a,b,c,d} 
        \rightarrow \widetilde{\K}$, 
        for $N \geq 7$.}
        \ENSURE{The isogeny $\varphi: \kgaudry_{a,b,c,d}\rightarrow \kgaudry_{a',b',c',d'}$
        with kernel generated by $R,S$.}
        \STATE $\psi = (\psi_X, \psi_Y, \psi_Z, \psi_T)$
        \STATE $a = \psi_X(a,b,c,d) \psi_X(b,a,d,c)
        \big(\psi_Y(a,b,c,d) \psi_Y(b,a,d,c)\big)^{-1}$
        \STATE $b = \psi_Z(a,b,c,d) \psi_Z(b,a,d,c)
        \big(\psi_T(a,b,c,d) \psi_T(b,a,d,c)\big)^{-1}$
        \STATE $s_a = \textsf{Sqrt}(a)$
        \STATE $s_b = \textsf{Sqrt}(b)$
        \STATE $\gamma = \psi_X(a,b,c,d) \psi_Z(d,c,b,a)
        \big(\psi_Y(a,b,c,d) \psi_T(d,c,b,a)\big)^{-1}$
        \FOR{$\alpha \in \{s_a, -s_a \}$}
         \FOR{$\beta \in \{s_b, -s_b \}$}
            \STATE \hspace{2ex} $\lambda_Y = \alpha$
            \STATE \hspace{2ex} $\lambda_Z = \alpha\beta$
            \STATE \hspace{2ex} $\lambda_T = \beta\gamma$
            \STATE \hspace{2ex} $\varphi = (\psi_X, \lambda_Y\psi_Y, \lambda_Z\psi_Z, 
            \lambda_T\psi_T)$\label{line:scalingsqrt:varphi}
            \IF{$\sigma_i\big(\varphi(X,Y,Z,T)\big) = 
            \varphi\big( \sigma_i(X,Y,Z,T)\big)$ for $i = 1,2,3$}
               \RETURN$\varphi$
            \ENDIF
         \ENDFOR
        \ENDFOR
        \RETURN $\bot$
    \end{algorithmic} 
\end{algorithm}

\begin{proposition}\label{prop:scalingsqrt}
   \Cref{alg:scalingsqrt} terminates in $\cost{sqrt}$, 
   where 
   $$\cost{sqrt} \leq 2 \ {\tt Sq} + 1 \ {\tt I} + 
   (10 + 12\ell) \ {\tt M},$$
   where $\ell \leq (N+1)(N+2)(N+3)/24$.
\end{proposition}
\begin{proof}
   \Cref{alg:scalingsqrt} requires the computation of $2$ square roots 
   in $K$, as well as $8$ $K$-mutliplications and inverses of 
   $\psi_Y(a,b,c,d) \psi_Y(b,a,d,c)$,
   $\psi_T(a,b,c,d) \psi_T(b,a,d,c)$ and $\psi_Y(a,b,c,d) \psi_T(d,c,b,a)$ in $K$.
   Using batched inversion~\cite[\S 10.3.1]{Montgomery}, we can compute these 
   $3$ inverses with $2$ multiplications and $1$ inverse in $K$. 
    
   In Line~\ref{line:scalingsqrt:varphi}, we use $3\ell$ {\tt M}
   to calculate 
   $(\psi_X, \lambda_Y\psi_Y, \lambda_Z\psi_Z, 
   \lambda_T\psi_T)$, where $\ell \leq (N+1)(N+2)(N+3)/24$,
   as described in the proof of~\Cref{prop:scalingGE}. We run this at most $4$ times, 
   thus requiring at most $12\ell$ 
   $K$-multiplications.
\end{proof}

\begin{remark}\label{remark:cost-of-sqrt}
    The most costly operation in~\Cref{alg:scalingsqrt} 
    is the computation of square roots using \textsf{Sqrt}. 
    When ${\rm{char}}(K) = p$, 
    and $K = \mathbb{F}_{p^m}$ is a finite field 
    (for some $m \in \mathbb{N}$), we compute square roots using the 
    Tonelli--Shanks algorithm~\cite{Tonelli,Shanks} using Scott's optimisation in~\cite{Scott}. This costs $2$ exponentiations and a few multiplications 
    and additions in $K$.

    When $K = \Q$, we can use, for example, the `Karatsuba Square Root'
    algorithm, as depicted by Brent and Zimmermann~\cite[Algorithm 1.12]{MCA},
    for computing square roots in $\Z$ (as 
    used in the GNU Multiple Precision Arithmetic Library~\cite{gmp})
    on the numerator and denominator of our rational number. This costs 
    $O(\frac{3}{2}M(N/2))$, where $M(n)$ is the time to multiply two numbers of $n$ limbs 
    using Karatsuba multiplication and $O(6M(N/2))$ when using FFT multiplication.
\end{remark}

Using the previous method depicted in~\Cref{alg:scalingsqrt}, it is 
not necessary to compute square roots in $K$ if the only thing required 
is to compute the constants $E', F', G', H' \in K$ 
appearing in the equation defining 
the image Kummer surface. Indeed, from $(b'/a')^2$, 
$(d'/c')^2$, $(c'/b')^2$, and
$a'c'/(b'd')$ we can compute $(a^2, b^2, c^2, d^2)$ (projectively). 
This leads us to the definition of~\Cref{alg:getimage}, which on input 
$\psi$ will output the constants $(E', F', G', H')$.
In this way, we only require the square roots in order to compute the isogeny $\varphi$, 
and therefore to push points through the isogeny. 

\Cref{alg:getimage} terminates in at most $34${\tt M}, 
$4${\tt S}, $1${\tt I}, and $20${\tt a}.
Note, here we are computing the $7$ inverses in
Lines~\ref{line:getimage:inverse1} to~\ref{line:getimage:inverse3},
Line~\ref{line:getimage:inverse4}, and
Lines~\ref{line:getimage:inverse5} to~\ref{line:getimage:inverse7} 
using batched inversions~\cite[\S 10.3.1]{Montgomery}.

\begin{algorithm}
	\caption{$\textsf{GetImage}$: find the constants in the equation defining the 
   image of an $(N,N)$-isogeny when $N\geq 7$}
   \label{alg:getimage}
	\begin{algorithmic}[1]
        \REQUIRE{A map of degree~$N$, $\psi: \kgaudry_{a,b,c,d} 
        \rightarrow \widetilde{\K}$, 
        for $N \geq 7$.}
        \ENSURE{The constants $(E', F', G', H')$ 
        in the equation defining $\kgaudry_{a',b',c',d'}$}
        \STATE $\psi = (\psi_1, \psi_2, \psi_3, \psi_4)$
        \STATE $\alpha_{21} = \psi_1(b,a,d,c)\psi_2(a,b,c,d)
        \big(\psi_1(a,b,c,d)\psi_2(b,a,d,c)\big)^{-1}$\label{line:getimage:inverse1}
         \STATE $\alpha_{43} = \psi_3(b,a,d,c)\psi_4(a,b,c,d
         )\big(\psi_3(a,b,c,d)\psi_4(b,a,d,c)\big)^{-1}$\label{line:getimage:inverse2}
         \STATE $\alpha_{32} = \psi_2(d,c,b,a)\psi_3(a,b,c,d)
         \big(\psi_2(a,b,c,d)\psi_3(d,c,b,a)\big)^{-1}$\label{line:getimage:inverse3}
         \STATE $\alpha_{31} = \alpha_{32}\cdot \alpha_{21}$
         \STATE $\alpha_{41} = \alpha_{43} \cdot \alpha_{31}$
         \STATE $\beta = \psi_1(a,b,c,d)\psi_2(d,c,b,a)\big(\psi_2(a,b,c,d)
         \psi_1(d,c,b,a)\big)^{-1}$\label{line:getimage:inverse4}
         \STATE $(a_2, b_2, c_2, d_2) = (1, \alpha_{21}, \alpha_{31}, \alpha_{41})$
         \STATE $A,B,C,D = \mathcal{H}(a_2, b_2, c_2, d_2)$
         \STATE $(a_4, b_4, c_4, d_4) = \mathcal{S}(a_2, b_2, c_2, d_2)$
         \STATE $\gamma_1 = (a_2d_2-b_2c_2)^{-1}$\label{line:getimage:inverse5}
         \STATE $\gamma_2 = (a_2c_2-b_2d_2)^{-1}$\label{line:getimage:inverse6}
         \STATE $\gamma_3 = (a_2b_2-c_2d_2)^{-1}$\label{line:getimage:inverse7}
         \STATE $E = \beta \gamma_1\gamma_2\gamma_3 b_2d_2ABCD$
         \STATE $F = \gamma_1\cdot (a_2^2 - b_2^2 - c^2_2 + d^2_2)$
         \STATE $G = \gamma_2 \cdot (a_2^2 - b_2^2 + c^2_2 - d^2_2)$
         \STATE $H = \gamma_3 \cdot (a_2^2 + b_2^2 - c^2_2 - d^2_2)$
        \RETURN $(E, F, G, H)$ 
    \end{algorithmic} 
\end{algorithm}

\subsubsection{Compute the image constants $(a', b',c',d')$}
We now have the $(N,N)$-isogeny $\varphi = \lambda \circ \psi$ where 
$$\lambda: (X,Y,Z,T) \mapsto 
(\lambda_X X, \lambda_Y Y, \lambda_Z Z, \lambda_T T),$$ 
is the scaling map obtained from the previous step. We compute
$(a', b',c',d')$ by evaluating the $(N,N)$-isogeny $\varphi$
at $(X,Y,Z,T) = (a,b,c,d)$. 
This costs at most $4\big( {{N+3}\choose{3}} + N + 1\big) = O(N^3)$
multiplications in $K$. Indeed, it costs $4(N+1)$ 
$K$-multiplications to compute all powers 
of $a, b, c, d$ up to $a^N, b^N, c^N, d^N$, and then 
at most $3{{N+3}\choose{3}}$ $K$-multiplications to evaluate all the 
monomials in $\varphi$ from these. Finally, it costs at most 
${{N+3}\choose{3}}$ $K$-multiplications to multiply these evaluated
monomials with their coefficients.  

\subsubsection{Complexity of finding an $(N,N)$-isogeny between 
Fast Kummer surfaces}\label{subsubsection:fullalgorithm-fast}

Putting together all steps of the algorithm we obtain 
\textsf{GetIsogeny}, given by~\Cref{alg:getisogeny}, which 
on input $N$-torsion points $R,S$ on Fast Kummer surface $\kgaudry_{a,b,c,d}$
generating a maximal isotropic subgroup of $N$-torsion group of $\kgaudry$, 
will output the
$(N,N)$-isogeny $\varphi$ with kernel $\langle R, S \rangle$. Here, we assume \Cref{conj:basis-space}, as otherwise the bases formed in Lines~\ref{line:findisogeny:findbasisR} and~\ref{line:findisogeny:findbasisS} may be wrong.
The costs $\cost{GE}$, $\cost{sqrt}$ 
in Line~\ref{line:findisogeny-cost} are as defined in~\Cref{prop:scalingGE} and~\Cref{prop:scalingsqrt}, 
respectively.

\begin{algorithm}
	\caption{$\textsf{GetIsogeny}$: find the isogeny $\varphi$ with kernel generated by $N$-torsion 
   points $R,S$ on $\kgaudry$ assuming~\Cref{conj:basis-space}}
   \label{alg:getisogeny}
	\begin{algorithmic}[1]
        \REQUIRE{$N$-torsion points $R,S$ on Fast Kummer surface $\kgaudry_{a,b,c,d}$}
        \ENSURE{The isogeny $\varphi: \kgaudry_{a,b,c,d}\rightarrow \kgaudry_{a',b',c',d'}$
        with kernel generated by $R,S$.}
        \STATE $\mathbf{R} = (2R, 3R, \dots, ((N-1)/2)R)$\label{line:getisogeny:scalarmultR}
        \STATE $\mathbf{S} = (2S, 3S, \dots, ((N-1)/2)S)$\label{line:getisogeny:scalarmultS}
        \STATE $\mathcal{I}_N = \{ (1,\dots,1), (1, \dots, 2), \dots, (2, \dots, 2),
        (3,\dots,3), (3, \dots, 4), \dots, (4, \dots, 4) \}$
        \STATE $B^{(1)}_R, B^{(2)}_R, B^{(3)}_R, B^{(4)}_R = 
        \textsf{FindBasis}(R, \mathbf{R}, \mathcal{I}_N)$\label{line:findisogeny:findbasisR}
        \STATE $B^{(1)}_S, B^{(2)}_S, B^{(3)}_S, B^{(4)}_S  = 
        \textsf{FindBasis}(S, \mathbf{S}, \mathcal{I}_N)$\label{line:findisogeny:findbasisS}
        \STATE Reduce forms in $B_R$ and $B_S$ modulo the Kummer equation 
        defining $\kgaudry$\label{line:findisogeny:reducing}
        \FOR{$i = 1$ to $4$}
            \STATE $\psi_i = \textsf{FindIntersection}(B^{(i)}_R, B^{(i)}_S)$\label{line:findisogeny:findintersection}
        \ENDFOR
        \STATE $\psi = (\psi_1, \psi_2, \psi_3, \psi_4)$
        \IF{$N = 5$}\label{line:findisogeny:startfifloop}
            \RETURN $\textsf{Scaling}_{5}(\psi)$
        \ELSIF{$\cost{GE} \leq \cost{sqrt}$}\label{line:findisogeny-cost}
            \RETURN $\textsf{Scaling}_{\text{GE}}(\psi)$
        \ELSE 
        \RETURN $\textsf{Scaling}_{\text{sqrt}}(\psi)$
        \ENDIF\label{line:findisogeny:endifloop}
        \RETURN $\bot$
    \end{algorithmic} 
\end{algorithm}

Asymptotically,~\Cref{alg:getisogeny} is dominated by the call \textsf{FindBasis} in Lines~\ref{line:findisogeny:findbasisR} and~\ref{line:findisogeny:findbasisS}. The cost of obtaining multiples of $R$, namely
$2R, \dots, ((N-1)/2)R$, in Line~\ref{line:getisogeny:scalarmultR} of~\Cref{alg:getisogeny}
is at most $\sum_{n = 2}^{(N-1)/2} 9\big\lceil \log_2(n) \big\rceil$ {\tt S} and 
$\sum_{n = 2}^{(N-1)/2} 16\big\lceil \log_2(n) \big\rceil$ {\tt M}~\cite[Theorem 3.6]{GaudryFastGenus2}. Furthermore, reducing the basis elements in $B_R$ and $B_S$ modulo the Kummer equation in 
Line~\ref{line:findisogeny:reducing} costs at most a handful of additions.

Our optimised implementation shows that in the case of Fast Kummer surfaces
$\cost{biquad} \leq 12{\tt S} + 43 {\tt M} + 25 {\tt a}$. 
Therefore, we have that the bottleneck step (asymptotically) costs 
$$\cost{basis} \leq (3^{(N+3)/2} + 4 \cdot 3^{(N-1)/2} - 27) {\tt M}_{\text{poly}} + 6(N-1){\tt S} + \frac{43}{2}(N-1) {\tt M} + \frac{25}{2}(N-1) {\tt a}$$ for such Kummer surface models.

\bigskip\bigskip\bigskip

\section{Implementation and Performance}\label{section6}
In this section, we investigate the performance of our
algorithms when applied to Fast Kummer surfaces, as 
implemented in~\cite{github}.
\par
We will fix $K = \mathbb{F}_{p^m}$ to be a finite field,
for some $m \in \mathbb{N}$.
To set up our experiments, we consider $p$ be a prime such that $2^4N | p+1$.
By restricting to superspecial Fast Kummer surfaces, which are 
therefore defined over $\mathbb{F}_{p^2}$, we can fix $m = 2$ and
our choice of prime ensures that we have full $K$-rational $2$-torsion and 
$K$-rational $N$-torsion that will generate the $(N,N)$-isogeny.
The 
experiments were run in {\tt MAGMA} V.2.25-6 on 
Intel(R) Core{\texttrademark} i7-1065G7 CPU @ 1.30GHz $\times$ 8
with $15.4$ GiB memory.

\begin{remark}
   Restricting our experiments to superspecial Kummer surfaces is 
   inspired by the setup of various cryptographic primitives 
   constructed in isogeny-based cryptography 
   (see, for example,~\cite{flynnti,hashfunction,CorteRealSantosCostelloSmithIsogenies}), and allows us to easily 
   obtain $\mathbb{F}_{p^2}$-rational $(N,N)$-isogenies. 
   We replicate this setup to demonstrate the efficiency of our algorithms, 
   but emphasise that the methods presented in this paper hold for more general fields $K$. 
   We remark, however, that the performance of our methods for large $N$ may be hindered by
   coefficient blow-up for these more general fields. 
\end{remark}

\subsection{Evaluating the scaling algorithms}\label{subsec:eval-scaling}

We first analyse Method (2) and (3) of finding the 
final scaling map described in~\Cref{subsubsec:scalingGE-fast,subsubsec:scalingsqrt-fast},
repsectively. In particular, we investigate how $\cost{GE}$
and $\cost{sqrt}$ vary with $N$ and $p$. This is necessary to 
precisely determine the condition in Line~\ref{line:findisogeny-cost}
of~\Cref{alg:getisogeny}.
\par
Using a cost metric of $1 {\tt I} = \log_2(p) {\tt M}$ and $1 {\tt a} = 0 {\tt M}$
for $K = \F_{p^2}$,~\Cref{prop:scalingGE} tells us that 
$$\cost{GE}
    \leq \frac{N^9}{41472} + O(N^8) \, {\tt M}
    $$
Additionally,~\Cref{prop:scalingsqrt} 
and~\Cref{remark:cost-of-sqrt} 
show that $\cost{sqrt}$ is approximately equal to $4$ exponentiations and $1$ inversion in 
$\F_{p^2}$, for which the cost is around $(5\log_2(p) + (N+1)(N+2)(N+3)/2 + 10)$ {\tt M}.  
From this, we expect that $\cost{GE} \leq \cost{sqrt}$ for $N$ and $p$ 
such that $$N^9 - 20736N^3 - 124416N^2 - 228096N - 539136 \leq 207360\log_2(p).$$ 
For example, when $N = 7$, we have that $\cost{GE} \leq \cost{sqrt}$ for $\log_2(p) \geq 121$.

To confirm the observations from our theoretical costs, 
we run two sets of experiments: 
(1) fix $N = 7$, and vary $p$, where
$\log_2(p)$ ranges from $12$ to $850$; 
(2) fix $p$ where $\log_2(p) = 100$, 
and vary $N$ from $7$ to $19$. The results are shown in~\Cref{fig:whichscaling-fixedN7}
and~\Cref{fig:whichscaling-fixedp}, respectively. 

 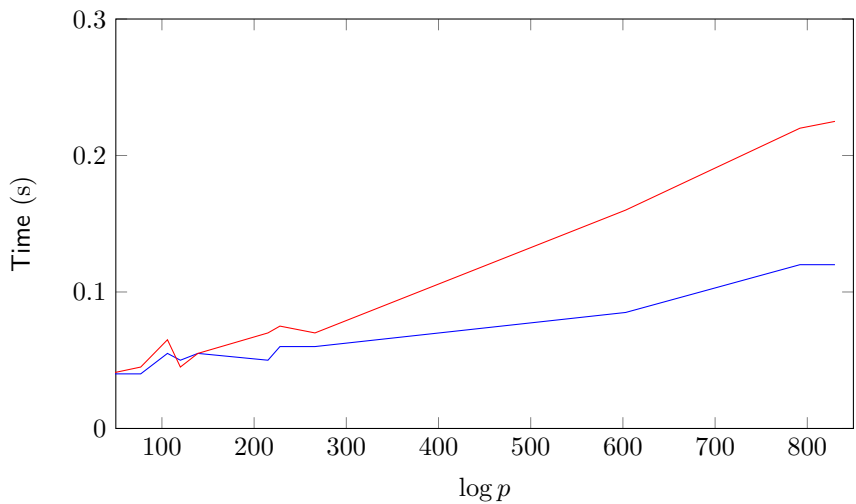
\begin{figure}[t]
     \centering
      \begin{tikzpicture}
          \begin{axis}[
             xlabel={$\log p$},
             ylabel={$\mathsf{Time}$ (s)},
             width=0.9\textwidth,
             height=200,
             ymin=000000, 
             ymax=0.3,
             xmin=50,
             xmax=850,
             xtick={0,100,...,850}
          ]
          
          \addplot[blue]                  table [mark=none, col sep=comma] {fixedN-method-2.txt};
          \addplot[red]                   table [mark=none, col sep=comma] {fixedN-method-3.txt};
          \end{axis}
  
      \end{tikzpicture}
     \caption{The time taken (in seconds) for \textsf{GetIsogeny} with method 2 of scaling $\textsf{Scaling}_{\text{GE}}$ 
     (in blue) and with method 3 of scaling$\textsf{Scaling}_{\text{sqrt}}$ (in red) 
     for a range of odd primes 
     $p$ and fixed prime $N = 7$. For each prime $p$, we average the time taken over $50$ runs.}\label{fig:whichscaling-fixedN7}
  \end{figure}
 
In~\Cref{fig:whichscaling-fixedN7}, we confirm that for $N = 7$, 
$\cost{GE} = \cost{sqrt}$ for $\log_2(p) \approx 140$, which is 
close to our theoretical estimate.
We remark that for larger $N$, this crossover point
will be larger. 

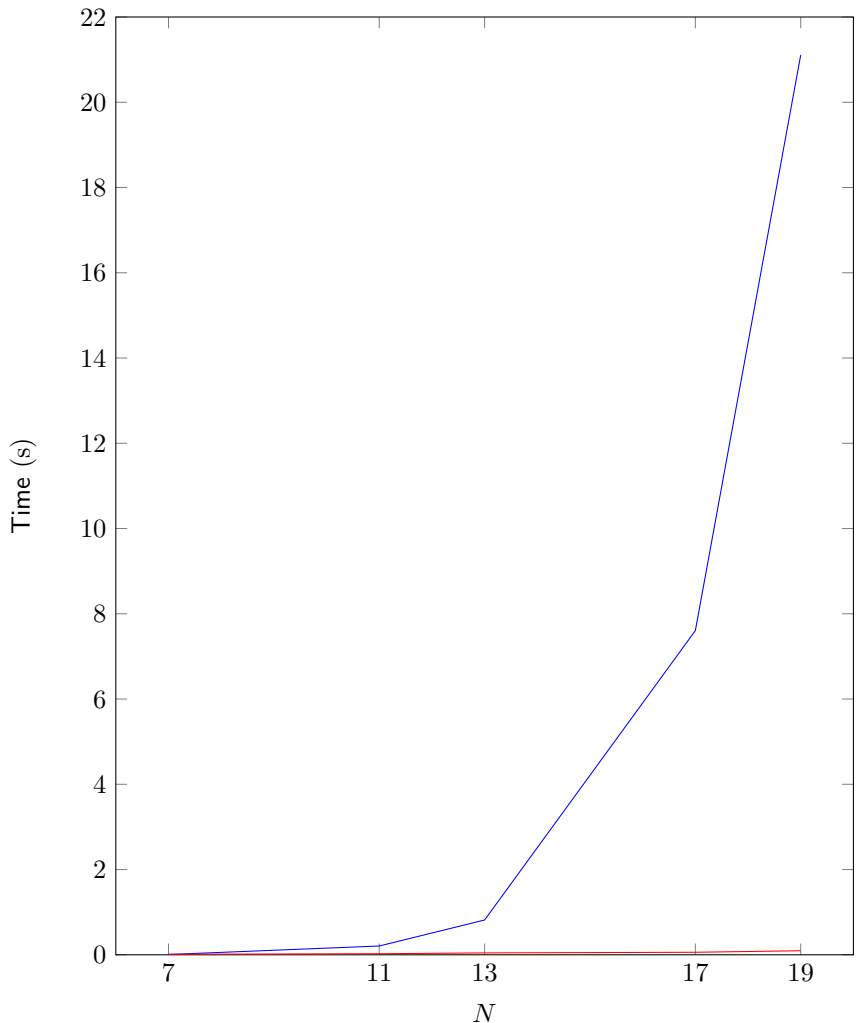
\begin{figure}[t]
   \centering
    \begin{tikzpicture}
        \begin{axis}[
            xlabel={$N$},
            ylabel={$\mathsf{Time}$ (s)},
            width=0.9\textwidth,
            height=400,
            ymin=0, 
            ymax=22,
            xmin=6,
            xmax=20,
            xtick={7,11,13,17,19}
        ]
        
        \addplot[blue]                  table [mark=none, col sep=comma] {fixedp-method-2.txt};
        \addplot[red]                   table [mark=none, col sep=comma] {fixedp-method-3.txt};
      \end{axis}

    \end{tikzpicture}  
   \caption{The time taken (in seconds) of method 2 of scaling $\textsf{Scaling}_{\text{GE}}$ 
   (in blue) and method 3 of scaling $\textsf{Scaling}_{\text{sqrt}}$ (in red) 
   for a range of odd primes $N$ and fixed prime $\log_2(p) \approx 100$. 
   For each $N$, we average the time taken over $50$ runs.}\label{fig:whichscaling-fixedp}
\end{figure}

In~\Cref{fig:whichscaling-fixedp} we see that the 
cost of $\textsf{Scaling}_{\text{sqrt}}$ increases at a 
slower rate with $N$ than the cost of 
$\textsf{Scaling}_{\text{GE}}$,
as predicted by our theoretical costs. 
We also ran experiments for $\log_2(p) \approx 500$ and observed
the same trend.

\subsection{Evaluation performance of algorithms}
In~\Cref{tab:results-getisogeny}, we give timings for different sections of~\Cref{alg:getisogeny} to compute an $(N,N)$-isogeny given 
$N$-torsion points 
$R,S$ generating kernel on domain Kummer surface $\kgaudry$ defined over 
$K$ for $N \geq 7$.

As expected, the sub-algorithm \textsf{FindBasis} (see~\Cref{alg:findbasis}) is the bottleneck step of 
\textsf{GetIsogeny} (see~\Cref{alg:getisogeny}). In comparison, the other sub-algorithms exhibit 
practical runtimes. Obtaining a method for finding the basis of spaces
$\spaceR$ and $\spaceS$ that scales polynomially with $N$ would have a large impact on the practicality of these 
methods for large $N$. 

\begin{table}[h!]
	\centering 
	\renewcommand{\tabcolsep}{0.08cm}
	\renewcommand{\arraystretch}{1.5}
   \begin{tabular}{cc|cccc|}
      \cline{3-6}
      \multicolumn{2}{l|}{}      & \multicolumn{4}{c|}{Time taken (s)}                                                                                                                                                                         \\ \hline
      \multicolumn{1}{|l|}{$N$}  & \multicolumn{1}{|l|}{$\lceil \log_2(p) \rceil$}  & \multicolumn{1}{c|}{\textsf{FindBasis}} & 
      \multicolumn{1}{c|}{\textsf{FindIntersection}} & 
      \multicolumn{1}{c|}{\begin{tabular}[c]{@{}c@{}}\textsf{Scaling}\\ (GE/sqrt for $N\geq 7$)\end{tabular}} & 
      \textsf{GetIsogeny}  \\ \hline
      \multicolumn{1}{|c|}{$5$} & \multicolumn{1}{|c|}{$104$}  & \multicolumn{1}{c|}{0.005}& \multicolumn{1}{c|}{0.001} 
      & \multicolumn{1}{c|}{0.000} & 0.010 \\
      \multicolumn{1}{|c|}{$7$}  & \multicolumn{1}{|c|}{$106$}& \multicolumn{1}{c|}{0.021}& \multicolumn{1}{c|}{0.002} 
      & \multicolumn{1}{c|}{0.007/0.012} & 0.045/0.050  \\
      \multicolumn{1}{|c|}{$11$} & \multicolumn{1}{|c|}{$95$}& \multicolumn{1}{c|}{0.721} & \multicolumn{1}{c|}{0.012} 
      & \multicolumn{1}{c|}{0.206/0.025} & 1.040/0.859  \\
      \multicolumn{1}{|c|}{$13$} & \multicolumn{1}{|c|}{$99$}& \multicolumn{1}{c|}{4.157} & \multicolumn{1}{c|}{0.025} 
      & \multicolumn{1}{c|}{0.816/0.044} & 5.238/4.466 \\
      \multicolumn{1}{|c|}{$17$} & \multicolumn{1}{|c|}{$94$}& \multicolumn{1}{c|}{85.267}  & \multicolumn{1}{c|}{0.048}  
      & \multicolumn{1}{c|}{7.605/0.059} & 93.838/86.292 \\
      \multicolumn{1}{|c|}{$19$}& \multicolumn{1}{|c|}{$105$} & \multicolumn{1}{c|}{416.329} & \multicolumn{1}{c|}{0.081} 
      & \multicolumn{1}{c|}{21.109/0.095}  & 439.464/418.450 \\ \hline
   \end{tabular}
	\vspace{0.3cm}
	\caption{Comparison of time taken for different sub-algorithms of \textsf{GetIsogeny}
   for various odd $N \geq 5$ using both scaling methods for $K = \mathbb{F}_{p^2}$
   with $\log_2(p) \approx 100$. We take the average over $50$ runs.}\label{tab:results-getisogeny}
\end{table}

In~\Cref{tab:results-getimage}, we give timings for \textsf{GetImage}
(\Cref{alg:getimage}), which computes the constants $E',F', G', H'$
in the defining equation of the image Kummer surface $\kgaudry_{a',b',c',d'}$ 
of the $(N,N)$-isogeny $\varphi$ with kernel generated by $N$-torsion points 
$R,S$,
given the degree-$N$ map $\psi: \kgaudry_{a,b,c,d} \rightarrow \widetilde{\K}$ (i.e., the $(N,N)$-isogeny 
before the final scaling).

\begin{table}[h!]
	\centering 
	\renewcommand{\tabcolsep}{0.08cm}
	\renewcommand{\arraystretch}{1.5}
   \begin{tabular}{|c|c|c|}
      \hline
      \multicolumn{1}{|l|}{$N$} & $\lceil \log_2(p) \rceil$ & Time taken for \textsf{GetImage} (s) \\ \hline
      \multirow{1}{*}{$7$}      & $106$  & 0.004 \\
      \multirow{1}{*}{$11$}     & $95$  & 0.014        \\
      \multirow{1}{*}{$13$}     & $99$ & 0.018          \\
      \multirow{1}{*}{$17$}     &  $94$ &  0.030  \\
      \multirow{1}{*}{$19$}     &  $105$  &  0.042 \\ \hline
      \end{tabular}
	\vspace{0.3cm}
	\caption{The time taken for \textsf{GetImage} for various odd $N \geq 7$. We fix the 
   base field $K = \mathbb{F}_{p^2}$ with $90 \leq \log_2(p) \leq 120$, 
   and average the time taken over $50$ runs.}\label{tab:results-getimage}
\end{table}

\subsection{Performance comparison with previous works}
We now compare the performance of our methods to the software 
package {\tt AVIsogenies v0.7}~\cite{AVIsogenies} for odd prime $N$.
For this comparison, we run the function {\tt IsogenieG2Theta.m}, which computes 
an isogeny from an abelian variety (where some precompuation is done). It takes as 
input a kernel in theta coordiantes (i.e., the points $R, S, R+S$ in theta 
coordinates) and outputs the \emph{theta null point} of the isogeneous Kummer surface.
For the fast Kummer surface $\kgaudry_{a,b,c,d}$, 
the theta null point is $(a, b, c, d)$, 
and thus this function is comparable to running our algorithm \textsf{GetIsogeny} 
(chosing the best scaling method for each $N$ and $p$ as per~\Cref{subsec:eval-scaling})
to obtain the $(N, N)$-isogeny 
$\varphi$ and then computing $\varphi((a,b,c,d)) = (a', b', c', d')$. 
\par
To compare our software with {\tt AVIsogenies} we therefore run the two algorithms 
described above for $N = 5, 7, 11, 13, 17$ for a fixed prime $p$ with $90 \leq \log_2(p) \leq 120$, 
taking the average time taken over $10$ runs. The results are given in~\Cref{tab:avisogenies-comparison}. We observe that our methods are 
comparable and 
outperform {\tt IsogenieG2Theta.m} for small $N = 5, 7, 11$. 
Furthermore, there is the benefit of our method that we
recover the explicit isogeny formul\ae{}, as well as the image 
constants $(a', b', c', d')$. 
For larger $N$, 
our algorithm is slower, mainly due to the fact that 
the subroutine \textsf{FindBasis} scales exponentially with $N$,
but we note again that our method also
recovers the explicit isogeny formul\ae{}, as well as the 
image constants $(a', b', c', d')$.
\par
We would like to emphasise that we do not intend timings to be
a complete description of the
performance of these algorithms; a fair comparison can only be made with 
precise operation counts. We merely present this broad comparison to 
showcase the 
interest in exploring these methods, and highlight important future work 
in improving the scalability of algorithm \textsf{FindBasis} 
(see~\Cref{alg:findbasis}).

\begin{table}[h!]
	\centering 
	\renewcommand{\tabcolsep}{0.08cm}
	\renewcommand{\arraystretch}{1.5}
   \begin{tabular}{cc|cc|}
      \cline{3-4}
      \multicolumn{2}{l|}{}      & \multicolumn{2}{c|}{Time taken (s)}   \\ \hline
      \multicolumn{1}{|l|}{$N$} & \multicolumn{1}{|l|}{$\big\lceil\log_2(p)\big\rceil$} 
       & \multicolumn{1}{c|}{This work} & 
      \multicolumn{1}{c|}{{\tt AVIsogenies}}  \\ \hline
      \multicolumn{1}{|c|}{$5$} & \multicolumn{1}{|c|}{$121$} & \multicolumn{1}{c|}{0.011}& \multicolumn{1}{c|}{0.022} \\
      \multicolumn{1}{|c|}{$7$} & \multicolumn{1}{|c|}{$120$} & \multicolumn{1}{c|}{0.060}& \multicolumn{1}{c|}{0.242} \\
      \multicolumn{1}{|c|}{$11$} &\multicolumn{1}{|c|}{$95$} & \multicolumn{1}{c|}{0.921} & \multicolumn{1}{c|}{1.402} \\
      \multicolumn{1}{|c|}{$13$} & \multicolumn{1}{|c|}{$98$} & \multicolumn{1}{c|}{4.769} & \multicolumn{1}{c|}{0.126} \\ 
      \multicolumn{1}{|c|}{$17$} & \multicolumn{1}{|c|}{$94$}  & \multicolumn{1}{c|}{95.060}  & \multicolumn{1}{c|}{0.226} \\ 
      \multicolumn{1}{|c|}{$19$} & \multicolumn{1}{|c|}{$108$}  & \multicolumn{1}{c|}{423.276}  & \multicolumn{1}{c|}{12.752}\\ 
      \hline
   \end{tabular}
	\vspace{0.3cm}
	\caption{The time taken for {\tt IsogenieG2Theta.m} in {\tt AVIsogenies} 
   and \textsf{GetIsogeny} with evaluation to compute the image theta constants 
   $(a', b', c', d')$ of an $(N, N)$-isogeny for various odd $N \geq 7$. We fix the 
   base field $K = \mathbb{F}_{p^2}$ with $90 \leq \log_2(p) \leq 125$, and 
   average the time taken over $50$ runs.}\label{tab:avisogenies-comparison}
\end{table}

\begin{remark}
   Unlike the algorithms presented in this paper, the complexity of the algorithms that we use from {\tt AVIsogenies} depends on $N \bmod 4$ (see, for example,~\cite[pg. 199]{LubiczRobertComputing}). This is shown by our experiments in~\Cref{tab:avisogenies-comparison}, which exhbits how $N = 11$ and $N = 19$ perform comparitively worse than $N = 13$ and $17$. 
\end{remark}

\end{document}